\documentclass[11pt,leqno]{article}
\usepackage{amsmath,amsthm}
\usepackage{amsfonts}
\usepackage{mathrsfs}
\usepackage{hyperref}
\usepackage{amssymb}
\usepackage{epsfig}
\usepackage{graphicx}
\usepackage{caption}
\usepackage{bm}
\numberwithin{equation}{section} 
\newtheorem{theorem}{\bf Theorem}[section]

\newtheorem{remark}{\bf Remark}[section]
\newtheorem{lemma}{\bf Lemma}[section]

\newcommand{\bu}{{\bf {u}}}
\newcommand{\bphi}{\mbox{\boldmath $\phi$}}
\newcommand{\bv}{{\bf {v}}}
\newcommand{\bF}{{\bf {F}}}
\newcommand{\bw}{{\bf {w}}}

\newcommand{\bz}{{\bf {z}}}
\newcommand{\f}[1]{\bf #1}
\newcommand{\norm}[1]{\lVert #1\rVert}

\newsavebox{\savepar}

\begin{document}

\title{Stabilization of Kelvin-Voigt viscoelastic fluid flow model}
\author{Sudeep Kundu,\; Amiya K. Pani\\
	Department of Mathematics\\
	IIT Bombay, Powai, Mumbai-400076 (India)\\
	Email: sudeep.kundu85@gmail.com, akp@math.iitb.ac.in .}
\maketitle
\abstract{In this article, stabilization result for the viscoelastic fluid flow problem governed by Kelvin-Voigt model, that is, convergence of the unsteady solution to a steady state solution is proved under the assumption that linearized self-adjoint steady state eigenvalue problem has a minimal positive eigenvalue. Both power and exponential convergence results are derived under various conditions 
on the forcing function. It is shown that results are valid uniformly in the time relaxation or some times called regularization parameter $\kappa$ as $\kappa\to 0$, which in turn, establishes results for the Navier-Stokes system.
}

\section{Introduction}
Consider the following equation arising in Kelvin-Voigt model of viscoelastic fluid flow problem: 
find $(\bu,p)$ such that
\begin{align}
  {\bu}_t-\kappa\Delta {\bu}_t-\nu\Delta {\bu}+{\bu}\cdot\nabla {\bu}+\nabla p&={\bf f}(x,t), \quad x\in \Omega, \quad t>0,\label{eq1.1}\\
 \nabla\cdot{\bu}&=0,\quad x\in \Omega, \quad t>0,\label{eq1.2}\\
 {\bu}(x,0)=u_0 \quad x\in \Omega, \quad {\bu}&=0 \quad \text{on} \quad\partial \Omega, \quad t\geq 0\label{eq1.3},
\end{align}
 where, $\nu>0$ is the coefficient of kinematic viscosity, $\kappa>0$ is the retardation time or the time of relaxation of deformations and $\Omega$ is a bounded convex polygonal domain in $\mathbb R^2$ with boundary $\partial \Omega$. Regarding viscoelastic fluid flow problems, Pavlovskii \cite{pavlovskii} first introduced this model as a model of weakly concentrated water-polymer mixture. Then, Oskolkov \cite{oskolkov1} and his collaborators called it as Kelvin-Voigt model. For applications of such models see \cite{burt1}, \cite{burt2} and \cite{cotter} and reference therein. Recently, Cao {\it {et al.}}\cite{clt} have proposed this model as an inviscid regularization with regularizing parameter $\kappa$ of the Navier-Stokes system.
 Since the system differs from the system of Navier-Stokes equations by $-\kappa\Delta\bu_t$, then one is curious to explore how far results on stabilization for the Navier-Stokes system carry over to the Kelvin-Voigt model \eqref{eq1.1}-\eqref{eq1.3}. Therefore, in this article, both  power and exponential convergence results of the unsteady solution to a steady state solution are proved under various assumptions on the forcing function ${\bf {f}}(x,t)$.\\
  For local and global solvability of the problem \eqref{eq1.1}-\eqref{eq1.3}, refer to \cite{oskolkov2}, \cite{oskolkov3}, \cite{oskolkov4} and \cite{oskolkov5}. Earlier, exponential decay property has been proved for the problem \eqref{eq1.1}-\eqref{eq1.3}
  when ${\bf {f}}=0$ in Bajpai {\it {et al.}}\cite{bnpdy}, but the constants appeared in their estimates depend on $(\frac{1}{\kappa^r}),\hspace{0.1cm} r\geq 1$. Subsequently, Pany {\it {et al.}}
  \cite{pbp} have modified the arguments of \cite{bnpdy} to establish exponential decay properties, which are valid uniformly in $\kappa$ as $\kappa\to 0$, that is, these results hold for the Navier-Stokes system. On related results on long term dynamics of the problem \eqref{eq1.1}-\eqref{eq1.3}, refer to Kalantarov and Titi \cite{ktv1}, Kalantarov \cite{ktv2} and Kalantarov {{\it et al.}} \cite{ktv3}.\\
 On stabilizability, Sobolevskii \cite{Sobolevskii} has shown exponential convergence of the unsteady solution of the Oldroyd's model to its steady state solution under the assumption that 
 forcing function ${\bf f}$ is H\"{o}lder continuous and exponentially decaying. Further, He {\it et al.} \cite{He} have shown both exponential and power convergence for the solution by relaxing the H\"{o}lder continuity of the forcing function and assuming that forcing function ${\bf f}$ has exponential or power decay property. For linearized viscoelastic flow problem asymptotic behavior is discussed in \cite{heli}.  \\
The main contribution of this article is on the convergence of the problem \eqref{eq1.1}-\eqref{eq1.3} to a steady state under the assumption that the associated linearized steady state self-adjoint eigenvalue problem has a minimal positive eigenvalue. Further, both exponential and power convergence results are shown for the velocity and  the pressure in various norms under a variety of assumptions on the forcing function. Moreover, it is proved that results are valid uniformly in the time relaxation or regularizing parameter $\kappa$ as $\kappa\to 0$. This, in turn, establishes that results are valid for the Navier-Stokes system.\\
For the rest of this article, first we introduce $\mathbb{R}^2$- valued function denoted by bold face type letters as
\begin{align*}
 {\bf H}_0^1 = (H_0^1(\Omega))^2, \quad {\bf L}^2 = (L^2(\Omega))^2 \quad
\text{and }\quad {\bf H}^m=(H^m(\Omega))^2,
\end{align*}
where $L^2(\Omega)$ is the space of square integrable functions defined 
in $\Omega $ with inner product $(\phi,\psi) =\displaystyle{\int_{\Omega}}\phi(x) 
\psi(x)\,dx $ and norm  $\|\phi\| = \left(\displaystyle{\int_{\Omega}} 
|\phi(x)|^2\,dx\right)^\frac{1}{2}$. Further, $H^m(\Omega)$ denotes the standard 
Hilbert Sobolev space of order $m\in \mathbb{N^+}$ with norm $\|\phi\|_m=
 \left(\displaystyle{\sum_{|\alpha| \leq m}}\displaystyle{\int_{\Omega}} 
|D^{\alpha}\phi|^2\,dx\right)^{1/2}$. Note that ${\bf H}^1_0$ is equipped 
with a norm 
$$\|\nabla\bv\|= \left(\displaystyle{\sum_{i,j=1}^{2}}
(\partial_j v_i, \partial_j v_i)\right)^{1/2}
=\left(\displaystyle{\sum_{i=1}^{2}}
(\nabla v_i, \nabla v_i)\right)^{1/2}.$$
Let ${\bf H}^{-1}$ be the dual space of ${\bf H}_0^1$ with norm $\norm{\cdot}_{-1}$. For more details see \cite{kesavan}.

The rest of the article is organized as follows. Section $2$ focuses on the corresponding steady state problem with some properties. Section $3$ is devoted to both exponential and power convergence result of unsteady solution to the corresponding steady state solution.
\section{Steady state problem and its properties }
In this section, first we introduce some spaces :
\begin{equation*}
 {\bf J}_1 = \{{\bphi} \in {\bf {H}}_0^1 : \nabla \cdot \bphi = 0\},
\end{equation*}

\begin{equation*}
 {\bf J}=\{{\bphi}\in {\bf L}^2: \nabla\cdot {\bphi}=0 \in \Omega,\quad {\bphi}.n=0 \quad\text{on}\quad \partial \Omega\}
\end{equation*}
and $H^m/\mathbb R$ be the quotient space with norm $\norm{p}_{H^m/\mathbb R}=\inf_{c\in \mathbb R}\norm{p+c}_m.$\\
Setting $-\tilde\Delta =-P\Delta:{\bf J}_1 \cap {\bf H}^2\subset {\bf J}\rightarrow {\bf J}$ as the stokes operator.
Note that the following Poincar\'{e} inequality holds true:
\begin{equation}\label{eqx1.1}
 \norm{\bv}^2\leq \dfrac{1}{\lambda_1}\norm{\nabla \bv}^2 \quad \forall \bv\in {\bf H}^1_0,
\end{equation}
where $\lambda_1^{-1}$ is the best possible constant depending on the domain $\Omega$.
Moreover, the following holds:
\begin{equation}\label{eqx1.2}
 \norm{\nabla v}^2\leq \dfrac{1}{\lambda_1}\norm{\tilde\Delta \bv}^2 \quad \forall \bv\in {\bf J}_1\cap {\bf H}^2.
\end{equation}
The continuous bilinear forms $a(\cdot,\cdot)$ and $d(\cdot,\cdot)$ on ${\bf H}_0^1\times {\bf H}_0^1$ and ${\bf H}_0^1\times{L}^2/\mathbb{R}$ are
$$a(\bu,\bv)=(\nabla\bu,\nabla\bv),\quad \forall\; \bu,\;\bv\in {\bf H}_0^1,$$
$$d(\bv,q)=(q,\nabla\cdot\bv),\quad \forall\; \bv\in {\bf H}_0^1,\; q\in {L}^2/\mathbb{R}, $$
and 
define the trilinear form $b(\cdot, \cdot, \cdot)$ as 
$$
b(\bv, \bw,\bphi):= \frac{1}{2} (\bv \cdot \nabla \bw , \bphi)
- \frac{1}{2} (\bv \cdot \nabla \bphi, \bw),\;\;\bv, \bw, \bphi \in {\bf H}_0^1.
$$
Note that for $\bv \in {\bf J}_1,$  $\bw, \bphi \in {\bf H}^1_0$; 
$ b(\bv, \bw,\bphi)= (\bv \cdot \nabla \bw , \bphi).$\\
The continuous bilinear form  $d(\bv, q)$ satisfy the property:
\begin{equation}\label{eqx1.14}
 c\norm{q}\leq \sup_{\bv\in {\bf H}^1_0} \frac{d(\bv, q)}{\norm{\nabla\bv}}, \quad \forall q\in {L}^2/\mathbb{R}.
\end{equation}
The trilinear form satisfy the following properties:
\begin{equation}\label{eqx2.1}
 b(\bu,\bv,\bw)=-b(\bu,\bw,\bv)\quad \forall \bu,\bv,\bw \in {\bf H}^1_0,
\end{equation}
\begin{equation}\label{eqx2.2}
b(\bu,\bv,\bw)\leq N\norm{\nabla \bu}\norm{\nabla \bv}\norm{\nabla \bw} \quad \forall \;  \bu,\bv,\bw \in {\bf H}^1_0,
\end{equation}
\begin{equation}\label{eqx2.3}
 b(\bu,\bv,\bw)\leq N\norm{\bu}^{1/2}\norm{\nabla\bu}^{1/2}\norm{\nabla \bv}^{1/2}\norm{\Delta\bv}^{1/2}\norm{\bw}\quad \forall \;\bu\in {\bf H}^1_0, \bv\in {\bf H}^2\cap {\bf H}^1_0, \bw\in {\bf L}^2 ,
\end{equation}
and
\begin{equation} \label{eq1.14}
b(\bv,\bw,\bw) =0\;\;\;\forall \;\bv,\bw\in {\bf J}_1.
\end{equation}
For more details we refer to \cite{raviart} and \cite{temam}.

 An equilibrium (steady state) solution $({\bu}^\infty,p^\infty)$ to the continuous problem \eqref{eq1.1}-\eqref{eq1.3} satisfies 
\begin{align}
 -\nu\Delta {\bu}^\infty+{\bu}^\infty\cdot\nabla {\bu}^\infty+\nabla p^\infty&={\bf f}^\infty \quad \text{in}\quad \Omega \label{eq1.4},\\
 \nabla\cdot{\bu}^\infty&=0 \quad \text{in}\quad \Omega \label{eq1.5},\\
 {\bu}^\infty&=0 \quad \text{on} \quad\partial \Omega \label{eq1.6},
\end{align}
where ${\bf f}^\infty=\lim_{t\to\infty}{\bf f}(t,x)$.
In its weak form, the steady state solution satisfies $(\bu^\infty,p^\infty)\in {\bf H}^1_0\times{ L}^2/\mathbb{R}$ 
\begin{align}
 \nu(\nabla\bu^\infty,\nabla\bphi)+(\bu^\infty\cdot\nabla\bu^\infty,\bphi)+(\nabla p^\infty,\bphi)&=({\bf f}^\infty,\bphi)\quad \forall\;\bphi\in{\bf H}^1_0\label{eq3.5},\\
 (\nabla\cdot\bu^\infty,\chi)=0\quad \forall\;\chi \in{\bf L}^2\label{eq3.4}.
\end{align}
Equivalently, seek $\bu^\infty\in {\bf J}_1$ such that
\begin{equation}\label{eq3.6}
\nu(\nabla\bu^\infty,\nabla\bphi)+(\bu^\infty\cdot\nabla\bu^\infty,\bphi)=({\bf f}^\infty,\bphi)\quad \forall\;\bphi\in {\bf J}_1.
\end{equation}
Now, we make the following assumption:\\
$\bf{(A1)}$ The eigenvalue problem
\begin{align}
 -\nu\Delta {\bf \bar z}+\frac{1}{2}[\nabla {\bu}^\infty+(\nabla {\bu}^\infty)^*]{\bf \bar z}+\bar\nabla q=\lambda {\bf \bar z}\quad \text{in}\;\Omega\notag,\\
 \nabla\cdot {\bf \bar z}=0 \quad \text{in}\; \Omega\notag,\\
 {\bf \bar z}=0 \quad\text{on} \quad \partial \Omega\label{eq1.10},
\end{align}
has a minimum eigenvalue $\lambda_0>0$. This is a sufficient condition for a unique solvability of the problem \eqref{eq1.4}-\eqref{eq1.6}. When ${\bf f}^\infty=0$, the solution $\bu^\infty$ of \eqref{eq3.5}-\eqref{eq3.4} becomes zero. Therefore, the eigenvalue problem \eqref{eq1.10} becomes standard Stokes problem and then the minimum eigenvalue $\lambda_0>0$. Moreover, from \eqref{eq3.6} 
\begin{align*}
\nu \norm{\nabla \bu^\infty}^2=({\bf f}^\infty, \bu^\infty)\leq \norm{{\bf f}^\infty}_{-1}\norm{\nabla \bu^\infty}
\end{align*}
and hence, $\norm{\nabla \bu^\infty}\leq \frac{1}{\nu}\norm{{\bf f}^\infty}_{-1}$.
When $\norm{{\bf f}^\infty}_{-1}$ is very small, with reasonable $\nu$, $\norm{\nabla \bu^\infty}$ is small. Thereby, as a consequence of Stokes eigenvalue problem, the assumption $\bf{(A1)}$ holds.\\
\begin{remark}
Now for the problem \eqref{eq1.4}-\eqref{eq1.6}, there is also another easily verifiable uniqueness condition namely
\begin{align}\label{nex1}
N\norm{{\bf f}^\infty}_{-1}\frac{1}{\nu^2}<1,
\end{align}
where, $N=\sup_{u, v, w \in H^1_0(\Omega)}\dfrac{(u\cdot \nabla v, w)}{\norm{\nabla u}\norm{\nabla v}\norm{\nabla w}}.$\\
Now, for ${\bf f}=0$, \eqref{nex1} clearly holds. When, ${\bf f}\neq 0$, \eqref{nex1} is valid under the smallness assumption on $\norm{\bu^\infty}_{1}$, which in fact can be obtained from the smallness assumption on $\norm{{\bf f}^\infty}_{-1}$. It can be proved that, if \eqref{nex1} is satisfied, then $\lambda_0>0$. For a proof see \cite{raviart}. For completeness, below we discuss how \eqref{eq1.10} and \eqref{nex1} are related. From \eqref{eq1.11}, it follows that
\begin{align}
\lambda\norm{{\bf \bar z}}^2=\nu \norm{\nabla{\bf \bar z}}^2-({\bf \bar z}\cdot\nabla{\bf \bar z},{\bu^\infty})\geq \nu \norm{\nabla{\bf \bar z}}^2-N\norm{\nabla {\bf \bar z}}^2\norm{\nabla {\bu^\infty}}&\geq \nu \norm{\nabla{\bf \bar z}}^2-N\norm{\nabla {\bf \bar z}}^2\frac{1}{\nu}\norm{f^\infty}_{-1}\notag\\
&=\nu \Big(1-N\norm{{\bf f}^\infty}_{-1}\frac{1}{\nu^2}\Big)\norm{\nabla{\bf \bar z}}^2\label{nex2}.
\end{align}
Hence, we can see if the right hand side of \eqref{nex2} is positive which follows
from \eqref{nex1}, minimal eigenvalue is also positive.
Also, $\lambda_0>0$ for other weaker assumption on $\bu^\infty$ or equivalently on forcing function namely when  $\norm{\rho_{\partial \Omega}{\bf f}^\infty}$, $\norm{\rho^2_{\partial \Omega}\frac{\partial {\bf f}^\infty}{\partial x_k}}$ are small, where $\rho_{\partial \Omega}(x)$ is a distance from point $x\in \Omega$ to boundary $\partial\Omega$. For more details, we refer to \cite{Sobolevskii}.
\end{remark}
Multiply \eqref{eq1.10} with $\bar\bz\in {\bf H}^2\cap {\bf H}^1_0$ with $\nabla\cdot\bz=0$ to obtain
\begin{equation}\label{eq1.11}
\nu\norm{\nabla {\bf \bar z}}^2+({\bf \bar z}\cdot\nabla {\bu}^\infty,{\bf \bar z})=\lambda\norm{{\bf \bar z}}^2,
\end{equation}
that is 
\begin{equation*}
 \nu \norm{\nabla{\bf \bar z}}^2-({\bf \bar z}\cdot\nabla{\bf \bar z},{\bu^\infty})=\lambda\norm{\bar\bz}^2.
\end{equation*}
So it follows that from \eqref{eq1.11} and \eqref{eqx2.1} the following inequality holds
\begin{align}
 \nu a(\bar\bz,\bar\bz)+b(\bar\bz,\bar\bz,\bar\bz)+b(\bar\bu^\infty,\bar\bz,\bar\bz)+b(\bar\bz,\bar\bu^\infty,\bar\bz)=\nu a(\bar\bz,\bar\bz)+b(\bar\bz,\bar\bu^\infty,\bar\bz)=\lambda\norm{\bar\bz}^2\geq \lambda_0\norm{\bar\bz}^2\label{eqxx1.3}.
\end{align}
Also it can be proved that 
\begin{equation}\label{eqx1.3}
\nu a(\bar\bz,\bar\bz)+b(\bar\bz,\bar\bz,\bar\bz)+b(\bar\bu^\infty,\bar\bz,\bar\bz)+b(\bar\bz,\bar\bu^\infty,\bar\bz)=\nu a(\bar\bz,\bar\bz)+b(\bar\bz,\bar\bu^\infty,\bar\bz)\geq \gamma_1\norm{\nabla\bar\bz}^2
\end{equation}
holds for some constant $\gamma_1>0.$ For a proof, see,  \cite{He} and \cite{Sobolevskii}.
Before proceeding further, we first discuss on some {\it{a priori}} bounds of the steady state solution, which are needed in the sequel.
\begin{remark}
 \begin{itemize}
  \item [(i)]Choose $\bphi=\bu^\infty$ in \eqref{eq3.6} to obtain
  \begin{equation*}
   \nu\norm{\nabla\bu^\infty}^2=({\bf f}^\infty,\bu^\infty)\leq \norm{{\bf f}^\infty}_{-1}\norm{\nabla \bu^\infty}.
  \end{equation*}
Therefore,  $\norm{\nabla\bu^\infty}$ is bounded by
\begin{equation}\label{eq3.7}
 \norm{\nabla\bu^\infty}\leq \frac{1}{\nu}\norm{f^\infty}_{-1}.
\end{equation} 
\item[(ii)] From the Poincar\'{e} inequality, it follows that 
\begin{equation}\label{eq3.8}
 \norm{\bu^\infty}\leq \dfrac{1}{\sqrt\lambda_1}\norm{\nabla\bu^\infty}\leq \dfrac{1}{\nu \sqrt\lambda_1}
 \norm{{\bf f}^\infty}_{-1}.
\end{equation}
Therefore, $\norm{u^\infty}$ is also bounded.
\item[(iii)]
Rewrite \eqref{eq3.6} as
\begin{equation}\label{eqnx3.6}
-\nu\tilde{\Delta}\bu^\infty+\bu^\infty\cdot\nabla\bu^\infty={\bf f}^\infty.
\end{equation}
 Forming the $L^2$- inner product between  \eqref{eqnx3.6} and  $-\tilde\Delta\bu^\infty$, we arrive at
 \begin{align*}
  \nu\norm{\tilde\Delta\bu^\infty}^2&=({\bf f}^\infty,-\tilde\Delta\bu^\infty)+(\bu^\infty\cdot\nabla\bu^\infty,\tilde\Delta\bu^\infty).
  \end{align*}
  Now use the Cauchy-Schwarz inequality, \eqref{eqx2.3} and the Young's inequality to obtain
 \begin{align*} 
   \nu\norm{\tilde\Delta\bu^\infty}^2
  &\leq \frac{\nu}{4}\norm{\tilde\Delta\bu^\infty}^2+\frac{1}{\nu}\norm{{\bf f}^\infty}^2+C\norm{\bu^\infty}^\frac{1}{2}\norm{\nabla\bu^\infty}\norm{\tilde\Delta{\bu^\infty}}^\frac{3}{2}\\
  &\leq \frac{\nu}{2}\norm{\tilde\Delta\bu^\infty}^2+C(\nu)\norm{{\bf f}^\infty}^2+C(\nu)\norm{\bu^\infty}^2\norm{\nabla\bu^\infty}^4.
 \end{align*}
 Hence, using \eqref{eq3.7}-\eqref{eq3.8} it follows that
 \begin{align*}
  \nu\norm{\tilde\Delta\bu^\infty}^2\leq C(\nu)(\norm{\f^\infty}^2+\norm{\bu^\infty}^2\norm{\nabla\bu^\infty}^4)\leq C,
 \end{align*}
 and, $\norm{\tilde\Delta\bu^\infty}$ is bounded.
\item[(iv)] From the interpolation inequality in $2$- dimension
\begin{equation}\label{eqnx1}
 \norm{\bu^\infty}_{L^\infty}\leq C\norm{\bu^\infty}^{1/2}\norm{\tilde\Delta\bu^\infty}^{1/2},
\end{equation}
and therefore, $\norm{\bu^\infty}_{L^\infty}$ is also bounded.
\item[(v)] From the Gagliardo-Nirenberg inequality in $2$- dimension
\begin{equation*}
 \norm{\bu^\infty}_{L^4}\leq C\norm{\bu^\infty}^{1/2}\norm{\nabla\bu^\infty}^{1/2},
\end{equation*}
and therefore, $\norm{\bu^\infty}_{L^4}$ is bounded and similarly, $\norm{\nabla\bu^\infty}_{L^4}$ is also bounded.
 \end{itemize}
\item[(vi)] From \eqref{eq3.5}, using \eqref{eqnx1}, it follows that
 \begin{align*}
 \norm{\nabla p^\infty}=\sup_{\bphi\in {\bf L}^2}\dfrac{|(\nabla p^\infty,\bphi)|}{\norm{\bphi}}\leq \norm{{\bf f}^\infty}+\nu
 \norm{\tilde\Delta\bu^\infty}+C\norm{\bu^\infty}^\frac{1}{2}\norm{\nabla\bu^\infty}\norm{\tilde\Delta\bu^\infty}^\frac{1}{2}.
 \end{align*}
 So, $\norm{\nabla p^\infty}$ is bounded and by using the Poincar\'{e} inequality, it follows that 
  
 $\norm{p^\infty}$ is bounded.
\end{remark}
\section{Stabilization result}
This section focuses on {\it{a priori}} bounds for the problem \eqref{eq1.12}, which are valid uniformly in time
  using both power and exponential weight functions in time. It is, further, shown  both the exponential and power convergence of $(\bu(t),p(t))$ to $(u^\infty,p^\infty)$, which is valid uniformly in $\kappa$ as $\kappa\to 0$.\\
Let ${\bz}={\bu}-{\bu}^\infty$, $q=p-p^\infty$.
Then, $({\bz},q)$ satisfies
\begin{align}
{\bz}_t-\kappa \Delta {\bz}_t-\nu\Delta{\bz}+{\bz}\cdot\nabla{\bz}+{\bu}^\infty\cdot\nabla {\bz}+{\bz}\cdot\nabla {\bu}^\infty+\nabla q&=\bF\label{eq1.7}\; \text{in}\; \Omega\\
\nabla\cdot{\bz}&=0 \label{eq1.8}\; \text{in}\; \Omega\\
{\bz}(x,0)={\bu}_0-{\bu}^\infty={\bz}_0 \text{(say)} \quad  \text{and} \quad \bz&=0 \quad \text{on} \quad \partial \Omega\label{eq1.9}.
\end{align}
  Now the weak formulation of \eqref{eq1.7}-\eqref{eq1.9} is to seek a pair of functions $(\bz(t),q(t))\in {\bf H}^1_0\times{ L}^2/\mathbb R$ with $\bz(0)=\bz_0$ such that $\forall t>0$
\begin{align}
 (\bz_t,\bphi)+\kappa a(\bz_t,\bphi)+\nu a(\bz,\bphi)+b(\bz,\bz,\bphi)+b(\bu^\infty,\bz,\bphi)+b(\bz,\bu^\infty,\bphi)+( q,\nabla\cdot\bphi)&=(\bF,\bphi)\quad \forall \;\bphi\in {\bf {H}}_0^1,\notag\\
 \qquad (\nabla \cdot\bz,\chi)&=0 \quad\forall\; \chi\in L^2\label{eq1.12}.
\end{align}
Equivalently, find $\bz(t)\in {\bf J}_{1}$ with $\bz(0)=\bz_0$ such that for $t>0$
\begin{align}
 (\bz_t,\bphi)+\kappa a(\bz_t,\bphi)+\nu a(\bz,\bphi)&+b(\bz,\bz,\bphi)+b(\bu^\infty,\bz,\bphi)\notag\\
 &+b(\bz,\bu^\infty,\bphi)=(\bF,\bphi)\quad \forall \bphi\in{\bf J}_1\label{eq1.13}.
\end{align}
Throughout this paper we always assume that
\begin{align}\label{eqx1.4}
 0<\alpha<\frac{\lambda_1}{4(1+\lambda_1\kappa)}\min \Big\{\nu,\gamma_1\Big\},\quad \delta_0>0 \quad \text{and} \quad \alpha_1=\alpha-\delta_0>0\quad,        \quad\beta=2\delta.  
\end{align}
$\tau(t)=\max\{\bar t,t\},\text{where}\quad \bar t=\dfrac{4\delta(1+\kappa\lambda_1)}{\lambda_1}\max\Big\{\frac{1}{\nu},\frac{1}{\gamma_1}\Big\}$ if  $\delta>0$\quad and $\tau(t)=1 \quad \text{if} \quad\delta=0$.\\
For showing convergence result regarding $\bz(t)$ only, we always assume that it is solution of \eqref{eq1.13} and for $(\bz(t),q(t))$ its a solution of \eqref{eq1.12}.
\begin{lemma}\label{lm1}
 Suppose $\bf{(A1)}$, $\limsup_{t\to \infty}\tau^\beta(t)e^{2\alpha_1 t}\norm{\bF(t)}^2\leq M$ and $\bz_0\in {\bf H}^1_0$ hold.
 Then,
 \begin{align*}
\tau^\beta(t)(\norm{e^{\alpha_1 t}\bz}^2+&\kappa\norm{e^{\alpha_1 t}\nabla\bz}^2)+\gamma_1e^{-2\delta_0 t}\int_{0}^{t}\tau^\beta(s)\norm{e^{\alpha s}\nabla\bz(s)}^2 ds \\
 &\leq e^{-2\delta_0 t}\tau^\beta(0)(\norm{\bz_0}^2+\kappa\norm{\nabla\bz_0}^2)+\frac{2}{\lambda_1\gamma_1}e^{-2\delta_0 t}\int_{0}^{t}\tau^\beta(s)\norm{e^{\alpha s}\bF}^2 ds,
 \end{align*}
where $\lambda_1>0$ is the minimum eigenvalue of the Dirichlet eigenvalue problem for the Laplace operator.
Moreover,
 \begin{equation}\label{eqx3.1}
 \limsup_{t\to\infty}\tau^\beta(t)(\norm{e^{\alpha_1 t}\bz}^2+\kappa\norm{e^{\alpha_1 t}\nabla\bz}^2)\leq \frac{M}{\lambda_1\gamma_1\delta_0}, 
 \end{equation}
and
\begin{equation}\label{eqx3.2}
 \limsup_{t\to\infty}\tau^\beta(t)\norm{e^{\alpha_1 t}\nabla\bz(t)}^2\leq \frac{M}{\lambda_1\gamma_1^2\delta_0}.
\end{equation}
\end{lemma}
\begin{proof}
 Set $\bphi=e^{2\alpha t}\bz$ in \eqref{eq1.13} to get
 \begin{align}
  \frac{1}{2}\frac{d}{dt}(\norm{e^{\alpha t}\bz}^2&+\kappa\norm{e^{\alpha t}\nabla\bz}^2)-\alpha(\norm{e^{\alpha t}\bz}^2+\kappa\norm{e^{\alpha t}\nabla \bz}^2)\notag\\
  &+e^{2\alpha t}\Big(\nu a(\bz,\bz)+b(\bz,\bz,\bz)+b(\bu^\infty,\bz,\bz)+b(\bz,\bu^\infty,\bz)\Big)=(e^{\alpha t}\bF,e^{\alpha t}\bz)\label{eq2.1}.
 \end{align}
Use \eqref{eqx1.3}, \eqref{eqx1.4}, the Poincar\'{e} inequality \eqref{eqx1.1} and the Young's inequality to obtain 
\begin{align}
 \frac{d}{dt}(\norm{e^{\alpha t}\bz}^2+\kappa\norm{e^{\alpha t}\nabla\bz}^2)+2\gamma_1\norm{e^{\alpha t}\nabla\bz}^2&\leq \frac{2\alpha}{\lambda_1}(1+\kappa\lambda_1)\norm{e^{\alpha t}\nabla\bz}^2+2(e^{\alpha t}\bF,e^{\alpha t}\bz)\notag\\
 &\leq\frac{\gamma_1}{2}\norm{e^{\alpha t}\nabla\bz}^2+\frac{2}{\lambda_1\gamma_1}\norm{e^{\alpha t}\bF}^2+\frac{\gamma_1}{2}\norm{e^{\alpha t}\nabla \bz}^2\label{eq2.2}.
\end{align}
Therefore, we arrive at
\begin{equation}\label{eq2.3}
\frac{d}{dt}(\norm{e^{\alpha t}\bz}^2+\kappa\norm{e^{\alpha t}\nabla\bz}^2)+\gamma_1\norm{e^{\alpha t}\nabla\bz}^2\leq \frac{2}{\lambda_1\gamma_1}\norm{e^{\alpha t}\bF}^2.
\end{equation}
Multiply \eqref{eq2.3} with $\tau^\beta(t)$ to obtain
\begin{align}
\frac{d}{dt}\Big(\tau^\beta(t)(\norm{e^{\alpha t}\bz}^2+&\kappa\norm{e^{\alpha t}\nabla\bz}^2)\Big)+\gamma_1\tau^\beta(t)\norm{e^{\alpha t}\nabla\bz}^2\notag\\
&\leq \frac{2}{\lambda_1\gamma_1}\tau^\beta(t)\norm{e^{\alpha t}\bF}^2+\frac{d}{dt}(\tau^\beta(t))(\norm{e^{\alpha t}\bz}^2+\kappa\norm{e^{\alpha t}\nabla\bz}^2)\label{eqx12.3}. 
\end{align}
Now for $\bar t\geq 2\beta\frac{(1+\kappa\lambda_1)}{\gamma_1\lambda_1}$, it follows that
\begin{align}
\frac{d}{dt}\tau^\beta(t)=0, \quad \forall t\geq 0 \quad\text{and}\quad \beta=0; \quad \frac{d}{dt}\tau^\beta(t)=0 \quad \forall 0<t<\bar t \quad\text{and}\quad \beta>0.\label{eqxxx2.3}
\end{align}
For $t>\bar t \geq 2\beta\frac{(1+\kappa\lambda_1)}{\gamma_1\lambda_1}$, apply the Poincar\'{e} inequality to obtain
\begin{align}
\frac{d}{dt}(\tau^\beta(t))(\norm{e^{\alpha t}\bz}^2+\kappa\norm{e^{\alpha t}\nabla\bz}^2)&\leq \beta\tau^{\beta-1}(t)(\frac{1}{\lambda_1}+\kappa)\norm{e^{\alpha t}\nabla\bz}^2\notag\\
&\leq \frac{\gamma_1}{2}\bar t\tau^{\beta-1}(t)\norm{e^{\alpha t}\nabla\bz}^2\leq \frac{\gamma_1}{2}\tau^\beta(t)\norm{e^{\alpha t}\nabla\bz}^2\label{eqxx2.3}.
\end{align}
Therefore, using \eqref{eqxx2.3} in \eqref{eqx12.3}, we arrive at
\begin{equation}\label{eqx2.4}
 \frac{d}{dt}\Big(\tau^\beta(t)(\norm{e^{\alpha t}\bz}^2+\kappa\norm{e^{\alpha t}\nabla\bz}^2)\Big)+\frac{\gamma_1}{2}\tau^\beta(t)\norm{e^{\alpha t}\nabla\bz}^2\leq \frac{2}{\lambda_1\gamma_1}\tau^\beta(t)\norm{e^{\alpha t}\bF}^2.
\end{equation}
Integrate \eqref{eqx2.4} with respect to $t$ from $0$ to $t$ and then, multiply the resulting inequality by $e^{-2\delta_0 t}$ to obtain
\begin{align}\label{eq2.4}
 \tau^\beta(t)(\norm{e^{\alpha_1 t}\bz}^2+&\kappa\norm{e^{\alpha_1 t}\nabla\bz}^2)+\frac{\gamma_1}{2}e^{-2\delta_0 t}\int_{0}^{t}\tau^\beta(s)\norm{e^{\alpha s}\nabla\bz(s)}^2 ds\notag\\
 &\leq e^{-2\delta_0 t}\tau^\beta(0)(\norm{\bz_0}^2+\kappa\norm{\nabla\bz_0}^2)+\frac{2}{\lambda_1\gamma_1}e^{-2\delta_0 t}\int_{0}^{t}\tau^\beta(s)\norm{e^{\alpha s}\bF}^2 ds.
\end{align}
A use of L'Hospital's rule yields
\begin{align}
\limsup_{t\to\infty}\tau^\beta(t)(\norm{e^{\alpha_1 t}\bz}^2+\kappa\norm{e^{\alpha_1 t}\nabla\bz}^2)&\leq \frac{2}{\lambda_1\gamma_1}\limsup_{t\to\infty}e^{-2\delta_0 t}\int_{0}^{t}\tau^\beta(s)\norm{e^{\alpha s}\bF}^2 ds\notag\\
&= \frac{1}{\lambda_1\gamma_1\delta_0}\limsup_{t\to\infty}\tau^\beta(t)e^{2\alpha_1 t}\norm{\bF}^2\label{eq2.5}\notag\\
&\leq \frac{M}{\lambda_1\gamma_1\delta_0},
\end{align}
and 
\begin{equation}\label{eq2.6}
\limsup_{t\to\infty}\frac{\gamma_1}{2}e^{-2\delta_0 t}\int_{0}^{t}\tau^\beta(s)\norm{e^{\alpha s}\nabla\bz(s)}^2 ds\leq \frac{M}{\lambda_1\gamma_1\delta_0},
\end{equation}
that is
\begin{equation}\label{eq2.7}
\limsup_{t\to\infty}\tau^\beta(t)\gamma_1e^{2\alpha_1 t}\norm{\nabla\bz(t)}^2\leq \frac{2M}{\lambda_1\gamma_1\delta_0}.
\end{equation}
This completes the proof.
\end{proof}
\begin{lemma}\label{lm2}
 Under the assumption $\bf{(A1)}$, let $\limsup_{t\to\infty}\tau^\beta(t)e^{2\alpha_1 t}\norm{\bF(t)}^2\leq M$ and $\bz_0\in {\bf H}^2\cap {\bf H}^1_0$. Then, there exists a positive 
 constant $C=C(N,\nu,\lambda_1,\norm{f^\infty}_{-1})$ such that
 \begin{align*}
\tau^\beta(t)(\norm{e^{\alpha_1 t}\nabla\bz}^2&+\kappa\norm{e^{\alpha_1 t}\tilde\Delta\bz}^2)+\frac{\nu}{2} e^{-2\delta_0 t}\int_{0}^{t}\tau^\beta(s)\norm{e^{\alpha s}\tilde\Delta\bz(s)}^2 ds \\
&\leq Ce^{-2\delta_0 t}\int_{0}^{t}\tau^\beta(s)\norm{e^{\alpha s}\bF}^2 ds+
 C e^{-2\delta_0 t}\tau^\beta(0)(\norm{\bz_0}^2+\kappa\norm{\nabla\bz_0}^2+\kappa\norm{\tilde\Delta\bz_0}^2)\\
 &+C e^{-2\delta_0 t}\int_{0}^{t}\norm{\bz(s)}^2\norm{\nabla\bz(s)}^2\norm{e^{\alpha s}\nabla\bz(s)}^2 ds.
 \end{align*}
Moreover,
 \begin{equation}\label{eqx3.3}
 \limsup_{t\to\infty}\tau^\beta(t)(\norm{e^{\alpha_1 t}\nabla\bz}^2+\kappa\norm{e^{\alpha_1 t}\tilde\Delta\bz}^2)\leq \frac{CM}{2\delta_0}, 
 \end{equation}
and
\begin{equation}\label{eqx3.4}
 \limsup_{t\to\infty}\tau^\beta(t)\norm{e^{\alpha_1 t}\tilde\Delta\bz(t)}^2\leq \frac{CM}{\nu\delta_0}.
\end{equation}
\end{lemma}
\begin{proof}
 Rewrite \eqref{eq1.13} as
 \begin{align}
 \bz_t-\kappa\tilde{\Delta}\bz_t-\nu\tilde{\Delta}\bz+\bz\cdot\nabla\bz+\bu^\infty\cdot\nabla \bz+\bz\cdot\nabla \bu^\infty=\bF \label{eqxx2.8}.
 \end{align}
 Form the $L^2$- inner product between \eqref{eqxx2.8} and  $-e^{2\alpha t}\tilde\Delta\bz$ to find that
 \begin{align}
  \frac{1}{2}\frac{d}{dt}(\norm{e^{\alpha t}\nabla\bz}^2+\kappa\norm{e^{\alpha t}\tilde\Delta\bz}^2)-&\alpha(\norm{e^{\alpha t}\nabla\bz}^2+\kappa\norm{e^{\alpha t}\tilde\Delta\bz}^2)+\nu\norm{e^{\alpha t}\tilde\Delta\bz}^2\notag\\
  &=(e^{\alpha t}\bF,-e^{\alpha t}\tilde\Delta\bz)+e^{2\alpha t}\Big(b(\bz,\bz,\tilde\Delta\bz)+b(\bu^\infty,\bz,\tilde\Delta\bz)+b(\bz,\bu^\infty,\tilde\Delta\bz)\Big).\label{eq2.8}
 \end{align}
The term $b(\bz,\bz,\tilde\Delta\bz)$ is bounded by
\begin{align}
 b(\bz,\bz,\tilde\Delta\bz)&\leq N\norm{\bz}^{1/2}\norm{\nabla\bz}\norm{\tilde\Delta\bz}^{3/2}\notag\\
 &\leq\frac{\nu}{12}\norm{\tilde\Delta\bz}^2+\frac{1}{4}(\frac{9}{\nu})^3\norm{\bz}^2\norm{\nabla\bz}^4\label{eq2.9}.
\end{align}
The term $b(\bu^\infty,\bz,\tilde\Delta\bz)+b(\bz,\bu^\infty,\tilde\Delta\bz)$ is bounded by
 \begin{align}
 b(\bu^\infty,\bz,\tilde\Delta\bz)+b(\bz,\bu^\infty,\tilde\Delta\bz)&\leq 2N(\frac{1}{\lambda_1})^\frac{1}{4}\norm{\nabla\bu^\infty}^{1/2}\norm{\tilde\Delta\bu^\infty}^{1/2}\norm{\nabla\bz}\norm{\tilde\Delta\bz}\notag\\
 &\leq\frac{\nu}{12}\norm{\tilde\Delta\bz}^2+\frac{12}{\nu\sqrt{\lambda_1}}N^2\norm{\nabla\bu^\infty}\norm{\tilde\Delta\bu^\infty}\norm{\nabla\bz}^2\label{eq2.10}.
 \end{align}
The term $(e^{\alpha t}\bF,-e^{\alpha t}\tilde\Delta\bz)$ is bounded by
\begin{equation}\label{eq2.11}
 (e^{\alpha t}\bF,-e^{\alpha t}\tilde\Delta\bz)\leq\frac{\nu}{12}\norm{e^{\alpha t}\tilde\Delta\bz}^2+\frac{3}{\nu}\norm{e^{\alpha t}\bF}^2.
\end{equation}
Therefore from \eqref{eq2.8} using \eqref{eq2.9}, \eqref{eq2.10}, \eqref{eq2.11} and the Poincar\'{e} inequality, we arrive at
\begin{align}
 \frac{d}{dt}(\norm{e^{\alpha t}\nabla\bz}^2+&\kappa\norm{e^{\alpha t}\tilde\Delta\bz}^2)+2\nu\norm{e^{\alpha t}\tilde\Delta\bz}^2\notag\\
 &\leq 2\alpha\frac{(1+\lambda_1\kappa)}{\lambda_1}\norm{e^{\alpha t}\tilde\Delta\bz}^2+\frac{\nu}{2}\norm{e^{\alpha t}\tilde\Delta\bz}^2+\frac{6}{\nu}\norm{e^{\alpha t}\bF}^2\notag\\
 &+\frac{24}{\nu\sqrt{\lambda_1}}N^2\norm{\nabla\bu^\infty}\norm{\tilde\Delta\bu^\infty}\norm{e^{\alpha t}\nabla\bz}^2+\frac{1}{2}(\frac{9}{\nu})^3\norm{\bz}^2\norm{\nabla\bz}^2\norm{e^{\alpha t}\nabla\bz}^2 \label{eq2.12}.
\end{align}
Consequently using \eqref{eqx1.4} it follows that
\begin{align}
 \frac{d}{dt}(\norm{e^{\alpha t}\nabla\bz}^2+&\kappa\norm{e^{\alpha t}\tilde\Delta\bz}^2)+\nu\norm{e^{\alpha t}\tilde\Delta\bz}^2\notag\\
 &\leq \frac{6}{\nu}\norm{e^{\alpha t}\bF}^2+\frac{24}{\nu\sqrt{\lambda_1}}N^2\norm{\nabla\bu^\infty}\norm{\tilde\Delta\bu^\infty}\norm{e^{\alpha t}\nabla\bz}^2+\frac{1}{2}(\frac{9}{\nu})^3\norm{\bz}^2\norm{\nabla\bz}^2\norm{e^{\alpha t}\nabla\bz}^2 \label{eq2.13}.
\end{align}
Now multiplying by $\tau^\beta(t)$ to \eqref{eq2.13}, we obtain
\begin{align*}
 \frac{d}{dt}\Big(\tau^\beta(t)(\norm{e^{\alpha t}\nabla\bz}^2+&\kappa\norm{e^{\alpha t}\tilde\Delta\bz}^2)\Big)+\nu\tau^\beta(t)\norm{e^{\alpha t}\tilde\Delta\bz}^2\notag\\
 &\leq \frac{6}{\nu}\tau^\beta(t)\norm{e^{\alpha t}\bF}^2+\frac{24}{\nu\sqrt{\lambda_1}}N^2\norm{\nabla\bu^\infty}\norm{\tilde\Delta\bu^\infty}\tau^\beta(t)\norm{e^{\alpha t}\nabla\bz}^2\\
 &+\frac{1}{2}(\frac{9}{\nu})^3\tau^\beta(t)\norm{\bz}^2\norm{\nabla\bz}^2\norm{e^{\alpha t}\nabla\bz}^2+\frac{d}{dt}(\tau^\beta(t))(\frac{1}{\lambda_1}+\kappa)\norm{e^{\alpha t}\tilde\Delta\bz}^2.
\end{align*}
For $\bar t\geq \frac{2\beta(1+\kappa\lambda_1)}{\nu\lambda_1}$, it follows that
\begin{align*}
\frac{d}{dt}(\tau^\beta(t))(\frac{1}{\lambda_1}+\kappa)\norm{e^{\alpha t}\tilde\Delta\bz}^2=\beta\tau^{\beta-1}(t)\frac{(1+\kappa\lambda_1)}{\lambda_1}\norm{e^{\alpha t}\tilde\Delta\bz}^2\leq \frac{\nu}{2}\bar t\tau^{\beta-1}(t)\norm{e^{\alpha t}\tilde\Delta\bz}^2\leq \frac{\nu}{2}\tau^\beta(t)\norm{e^{\alpha t}\tilde\Delta\bz}^2. 
\end{align*}

Integrate from $0$ to $t$ and then multiply the resulting inequality by $e^{-2\delta_0 t}$ to arrive at with a use of Lemma \ref{lm1}
\begin{align}
 \tau^\beta(t)(\norm{e^{\alpha_1 t}\nabla\bz}^2&+\kappa\norm{e^{\alpha_1 t}\tilde\Delta\bz}^2)+\frac{\nu}{2} e^{-2\delta_0 t}\int_{0}^{t}\tau^\beta(s)\norm{e^{\alpha s}\tilde\Delta\bz}^2\notag\\
 &\leq Ce^{-2\delta_0 t}\int_{0}^{t}\tau^\beta(s)\norm{e^{\alpha s}\bF}^2 ds+
  C e^{-2\delta_0 t}\tau^\beta(0)(\norm{\bz_0}^2+(1+\kappa)\norm{\nabla\bz_0}^2+\kappa\norm{\tilde\Delta\bz_0}^2)\notag\\
 &+C e^{-2\delta_0 t}\int_{0}^{t}\norm{\bz(s)}^2\norm{\nabla\bz(s)}^2\norm{e^{\alpha s}\nabla\bz(s)}^2 ds\label{eqx3.33}.
\end{align}
Now apply L'Hospital's rule and put $\alpha_1=0$ in \eqref{eqx3.1} and \eqref{eqx3.2} to obtain
\begin{align*}
 \limsup_{t\to\infty}\tau^\beta(t)(\norm{e^{\alpha_1 t}\nabla\bz}^2&+\kappa\norm{e^{\alpha_1 t}\tilde\Delta\bz}^2)\\
 &\leq C\frac{1}{2\delta_0}(\limsup_{t\to\infty}\tau^\beta(t)\norm{e^{\alpha_1 t}\bF}^2+\limsup_{t\to\infty}\tau^\beta(t)\norm{\bz(t)}^2\norm{\nabla\bz(t)}^2\norm{e^{\alpha_1 t}\nabla\bz(t)}^2)\\
 &\leq C\frac{1}{2\delta_0}\limsup_{t\to\infty}\tau^\beta(t)\norm{e^{\alpha_1 t}\bF}^2,
\end{align*}
and
\begin{align*}
 \limsup_{t\to\infty}\frac{\nu}{2\delta_0}\norm{e^{\alpha_1 t}\tilde\Delta\bz(t)}^2\leq C\frac{1}{2\delta_0}\limsup_{t\to\infty}\tau^\beta(t)\norm{e^{\alpha_1 t}\bF}^2.
\end{align*}
This concludes the proof.
\end{proof}
\begin{lemma}\label{lm3}
 Under the assumption $\bf{(A1)}$, let $\limsup_{t\to\infty}\tau^\beta(t)e^{2\alpha_1 t}\norm{\bF(t)}^2\leq M$ and $\bz_0\in {\bf H}^2\cap {\bf H}^1_0$. Then, there exists a positive 
 constant $C=C(N,\nu,\lambda_1,\norm{f^\infty}_{-1})$ such that
 \begin{align*}
 \tau^\beta(t)\norm{e^{\alpha_1 t}\nabla\bz}^2&+e^{-2\delta_0 t}\int_{0}^{t}\tau^\beta(s)(\norm{e^{\alpha s}\bz_t(s)}^2+2\kappa\norm{e^{\alpha s}\nabla\bz_t(s)}^2) ds\\
 &\leq C e^{-2\delta_0 t}\tau^\beta(0)(\norm{\bz_0}^2+\kappa\norm{\nabla\bz_0}^2+\kappa\norm{\tilde\Delta\bz_0}^2)\\
 &+C e^{-2\delta_0 t}\int_{0}^{t}\tau^\beta(s)\norm{e^{\alpha s}\bF}^2 ds+Ce^{-2\delta_0 t}\int_{0}^{t}\tau^\beta(s)\norm{\bz(s)}^2\norm{\nabla\bz(s)}^2\norm{e^{\alpha s}\nabla\bz(s)}^2 ds.
 \end{align*}
Moreover,
\begin{equation}\label{eqx3.7}
 \limsup_{t\to\infty} \tau^\beta(t)(\norm{e^{\alpha_1 t}\nabla\bz}^2)\leq CM\frac{1}{2\delta_0},
\end{equation}
and
\begin{equation}\label{eqx3.8}
\limsup_{t\to\infty}\tau^\beta(t)(\norm{e^{\alpha_1 t}\bz_t(t)}^2+2\kappa\norm{e^{\alpha_1 t}\nabla\bz_t(t)}^2)\leq CM.
\end{equation}
\end{lemma}
\begin{proof}
 Set $\bphi=e^{2\alpha t}\bz_t$ in \eqref{eq1.13} to obtain
 \begin{align}
  \norm{e^{\alpha t}\bz_t}^2+\kappa\norm{e^{\alpha t}\nabla\bz_t}^2+\frac{\nu}{2}\frac{d}{dt}\norm{e^{\alpha t}\nabla\bz}^2&=\nu\alpha\norm{e^{\alpha t}\nabla\bz}^2+
  (e^{\alpha t}\bF,e^{\alpha t}\bz_t)-e^{\alpha t}b(\bz,\bz,e^{\alpha t}\bz_t)\notag\\
  &-b(\bu^\infty,e^{\alpha t}\bz,e^{\alpha t}\bz_t)-b(e^{\alpha t}\bz,\bu^\infty,e^{\alpha t}\bz_t).\label{eq2.20}
 \end{align}
The term $(e^{\alpha t}\bF,e^{\alpha t}\bz_t)$ is estimated as
\begin{align*}
(e^{\alpha t}\bF,e^{\alpha t}\bz_t)\leq\frac{3}{2}\norm{e^{\alpha t}\bF}^2+\frac{1}{6}\norm{e^{\alpha t}\bz_t}^2.
\end{align*}
The term $e^{\alpha t}b(\bz,\bz,e^{\alpha t}\bz_t)$ is bounded by
\begin{align*}
e^{\alpha t}b(\bz,\bz,e^{\alpha t}\bz_t)&\leq e^{\alpha t}\norm{\bz}_{L^4}\norm{\nabla\bz}_{L^4}\norm{e^{\alpha t}\bz_t}\\
&\leq Ce^{\alpha t}\norm{\bz}^{1/2}\norm{\nabla\bz}\norm{\tilde\Delta\bz}^{1/2}\norm{e^{\alpha t}\bz_t}\\
&\leq \frac{1}{6}\norm{e^{\alpha t}\bz_t}^2+\frac{3C^2}{2}\norm{\bz}\norm{\nabla\bz}\norm{e^{\alpha t}\nabla\bz}\norm{e^{\alpha t}\tilde\Delta\bz}\\
&\leq \frac{1}{6}\norm{e^{\alpha t}\bz_t}^2+C \norm{\bz}^2\norm{\nabla\bz}^2\norm{e^{\alpha t}\nabla\bz}^2+C\norm{e^{\alpha t}\tilde\Delta\bz}^2.
\end{align*}
For the term $b(\bu^\infty,e^{\alpha t}\bz,e^{\alpha t}\bz_t)+b(e^{\alpha t}\bz,\bu^\infty,e^{\alpha t}\bz_t)$, we note that
\begin{align*}
b(\bu^\infty,e^{\alpha t}\bz,e^{\alpha t}\bz_t)+b(e^{\alpha t}\bz,\bu^\infty,e^{\alpha t}\bz_t)&\leq 2N\frac{1}{\lambda_1^4}
\norm{\nabla\bu^\infty}^{1/2}\norm{\tilde\Delta\bu^\infty}^{1/2}\norm{e^{\alpha t}\nabla\bz}\norm{e^{\alpha t}\bz_t}\\
&\leq\frac{1}{6}\norm{e^{\alpha t}\bz_t}^2+6N^2\frac{1}{\lambda_1^2}\norm{\nabla\bu^\infty}\norm{\tilde\Delta\bu^\infty}\norm{e^{\alpha t}\nabla\bz}^2.
\end{align*}
On substitution, we arrive at from \eqref{eq2.20} that
 \begin{align*}
 \norm{e^{\alpha t}\bz_t}^2+&2\kappa \norm{e^{\alpha t}\nabla\bz_t}^2+\nu\frac{d}{dt}\norm{e^{\alpha t}\nabla\bz}^2\notag\\
 &\leq 3\norm{e^{\alpha t}\bF}^2
 +C(\norm{e^{\alpha t}\nabla\bz}^2+\norm{e^{\alpha t}\tilde\Delta\bz}^2)+ C\norm{\bz}^2\norm{\nabla\bz}^2\norm{e^{\alpha t}\nabla\bz}^2.
 \end{align*}
 Multiply the above inequality by $\tau^\beta(t)$ to arrive at
 \begin{align*}
  \nu\frac{d}{dt}(\tau^\beta(t)\norm{e^{\alpha t}\nabla\bz}^2)&+\tau^\beta(t)(\norm{e^{\alpha t}\bz_t}^2+2\kappa \norm{e^{\alpha t}\nabla\bz_t}^2)\\
  &\leq 3\tau^\beta(t)\norm{e^{\alpha t}\bF}^2+C\tau^\beta(t)(\norm{e^{\alpha t}\nabla\bz}^2+\norm{e^{\alpha t}\tilde\Delta\bz}^2)\\
  &+C\tau^\beta(t)\norm{\bz}^2\norm{\nabla\bz}^2\norm{e^{\alpha t}\nabla\bz}^2.
 \end{align*}
Integrate above inequality from $0$ to $t$ and then,  multiply the resulting inequality by $e^{-2\delta_0 t}$ to obtain
\begin{align}
\tau^\beta(t)\norm{e^{\alpha_1 t}\nabla\bz}^2+&e^{-2\delta_0 t}\int_{0}^{t}\tau^\beta(s)(\norm{e^{\alpha s}\bz_t(s)}^2+2\kappa\norm{e^{\alpha s}\nabla\bz_t(s)}^2) ds\notag\\
&\leq e^{-2\delta_0 t}\tau^\beta(0)\norm{\nabla\bz_0}^2+Ce^{-2\delta_0 t}\int_{0}^{t}\tau^\beta(s)(\norm{e^{\alpha s}\nabla\bz(s)}^2+\norm{e^{\alpha s}\tilde\Delta\bz(s)}^2) ds\notag\\
&\quad+ Ce^{-2\delta_0 t}\int_{0}^{t}\tau^\beta(s)\norm{\bz(s)}^2\norm{\nabla\bz(s)}^2\norm{e^{\alpha s}\nabla\bz(s)}^2 ds\label{eqx3.5}.
\end{align}
An application of Lemmas \ref{lm1} and \ref{lm2} in \eqref{eqx3.5} yields
\begin{align}
\tau^\beta(t)\norm{e^{\alpha_1 t}\nabla\bz}^2&+e^{-2\delta_0 t}\int_{0}^{t}(\norm{e^{\alpha s}\bz_t(s)}^2+2\kappa\norm{e^{\alpha s}\nabla\bz_t(s)}^2) ds\notag\\
&\leq C e^{-2\delta_0 t}(\norm{\bz_0}^2+\kappa\norm{\nabla\bz_0}^2+\kappa\norm{\tilde\Delta\bz_0}^2)+Ce^{-2\delta_0 t}\int_{0}^{t}\tau^\beta(s)\norm{e^{\alpha s}\bF}^2 ds\notag\\
&\quad+C e^{-2\delta_0 t}\int_{0}^{t}\tau^\beta(s)\norm{\bz(s)}^2\norm{\nabla\bz(s)}^2\norm{e^{\alpha s}\nabla\bz(s)}^2 ds\label{eqx3.6}. 
\end{align}
Now as $t\to\infty$, using L'Hospital's rule, we obtain from \eqref{eqx3.6} as in Lemmas \ref{lm1} and \ref{lm2}
\begin{equation*}
 \limsup_{t\to\infty} \tau^\beta(t)(\norm{e^{\alpha_1 t}\nabla\bz}^2)\leq CM\frac{1}{2\delta_0},
\end{equation*}
\begin{equation*}
 \limsup_{t\to\infty} e^{-2\delta_0 t}\int_{0}^{t}\tau^\beta(s)(\norm{e^{\alpha s}\bz_t(s)}^2+2\kappa\norm{e^{\alpha s}\nabla\bz_t(s)}^2) ds\leq CM\frac{1}{2\delta_0}.
\end{equation*}
Hence,
\begin{equation*}
\limsup_{t\to\infty}\tau^\beta(t)(\norm{e^{\alpha_1 t}\bz_t(t)}^2+2\kappa\norm{e^{\alpha_1 t}\nabla\bz_t(t)}^2)\leq CM.
\end{equation*}
This completes the rest of the proof.
\end{proof}
\begin{lemma}\label{lm4}
 Under the assumption $\bf{(A1)}$, let $\limsup_{t\to \infty}\tau^\beta(t)e^{2\alpha_1 t}(\norm{\bF(t)}^2+\norm{\bF_t}^2_{-1})\leq M_1$
and  $\bz_0\in {\bf H}^2\cap {\bf H}^1_0$. Then, there exists a positive 
 constant $C=C(N,\nu,\lambda_1,\norm{f^\infty}_{-1})$ such that
 \begin{align*}
 \tau^\beta(t)(\norm{e^{\alpha_1 t}\bz_t}^2&+\kappa\norm{e^{\alpha_1 t}\nabla\bz_t}^2)+\nu e^{-2\delta_0 t}\int_{0}^{t}\tau^\beta(s) \norm{e^{\alpha s}\nabla\bz_t(s)}^2 ds\\ 
 &\leq C e^{-2\delta_0 t}\tau^\beta(0)(\norm{\bF_0}^2+(1+\kappa)\norm{\tilde\Delta\bz_0}^2)+Ce^{-2\delta_0 t}\int_{0}^{t}\tau^\beta(s) \norm{e^{\alpha s}\bF_t(s)}^2_{-1} ds\\
 &\quad+Ce^{-2\delta_0 t}\int_{0}^{t}\norm{\nabla\bz(s)}^2\norm{e^{\alpha s}\bz_t(s)}^2 ds+Ce^{-2\delta_0 t}\int_{0}^{t}\norm{e^{\alpha s}\bz_t(s)}^2 ds.
 \end{align*}
Moreover,
\begin{align}
 \limsup_{t\to\infty} \tau^\beta(t)(\norm{e^{\alpha_1 t}\bz_t}^2+\kappa\norm{e^{\alpha_1 t}\nabla\bz_t}^2)\leq C(M_1)\frac{1}{2\delta_0}\label{eqx3.10},
\end{align}
and
\begin{equation}\label{eqx3.11}
\limsup_{t\to\infty} \tau^\beta(t)\norm{e^{\alpha_1 t}\nabla\bz_t(t)}^2\leq C\frac{(M_1)}{\nu}.
\end{equation} 
\end{lemma}
\begin{proof}
 Differentiate \eqref{eq1.13} with respect to $t$ and set $\bphi=e^{2\alpha t}\bz_t$ to obtain
 \begin{align}
  (\bz_{tt},e^{2\alpha t}\bz_t)+&\kappa(\nabla\bz_{tt},e^{2\alpha t}\nabla\bz_t)+\nu(\nabla\bz_t,e^{2\alpha t}\nabla\bz_t)\notag\\
  &=((\bz\cdot\nabla\bz)_t,-e^{2\alpha t}\bz_t)-e^{2\alpha t}b(\bu^\infty,\bz_t,\bz_t)-e^{2\alpha t}b(\bz_t,\bu^\infty,\bz_t)+(\bF_t,e^{2\alpha t}\bz_t)\label{eq2.31}.
 \end{align}
Now \eqref{eq2.31} can be written as 
\begin{align}
 \frac{d}{dt}(\norm{e^{\alpha t}\bz_t}^2+&\kappa\norm{e^{\alpha t}\nabla\bz_t}^2)+2\nu\norm{e^{\alpha t}\nabla\bz_t}^2\notag\\
 &=2\alpha(\norm{e^{\alpha t}\bz_t}^2+\kappa\norm{e^{\alpha t}\nabla\bz_t}^2)+2(e^{\alpha t}\bF_t,e^{\alpha t}\bz_t)-2b(e^{\alpha t}\bz_t,\bz,e^{\alpha t}\bz_t)-2b(\bz,e^{\alpha t}\bz_t,e^{\alpha t}\bz_t)\notag\\
 &-2(\bu^\infty,e^{\alpha t}\bz_t,e^{\alpha t}\bz_t)-2b(e^{\alpha t}\bz_t,\bu^\infty,e^{\alpha t}\bz_t).\label{eq2.32}
\end{align}
The right hand side terms of \eqref{eq2.32} are bounded by
\begin{align*}
 2\alpha(\norm{e^{\alpha t}\bz_t}^2+\norm{e^{\alpha t}\nabla\bz_t}^2)\leq\frac{2\alpha(1+\kappa\lambda_1)}{\lambda_1}\norm{e^{\alpha t}\nabla\bz_t}^2\leq\frac{\nu}{2}\norm{e^{\alpha t}\nabla\bz_t}^2,
\end{align*}
\begin{equation*}
 2(e^{\alpha t}\bF_t,e^{\alpha t}\bz_t)\leq \frac{\nu}{6}\norm{e^{\alpha t}\nabla\bz_t}^2+\frac{6}{\nu}\norm{e^{\alpha t}\bF_t}^2_{-1},
\end{equation*}
\begin{align*}
 2b(e^{\alpha t}\bz_t,\bz,e^{\alpha t}\bz_t)+2b(\bz\cdot e^{\alpha t}\nabla\bz_t,e^{\alpha t}\bz_t)=2b(e^{\alpha t}\bz_t\cdot\nabla\bz,e^{\alpha t}\bz_t)&\leq C\norm{\nabla\bz}\norm{e^{\alpha t}\bz_t}\norm{e^{\alpha t}\nabla\bz_t}\\
 &\leq \frac{\nu}{6}\norm{e^{\alpha t}\nabla\bz_t}^2+C(\nu)\norm{e^{\alpha t}\bz_t}^2\norm{\nabla\bz}^2,
\end{align*}
and 
\begin{align*}
 2b(\bu^\infty, e^{\alpha t}\bz_t,e^{\alpha t}\bz_t)+b(e^{\alpha t}\bz_t,\bu^\infty,e^{\alpha t}\bz_t)=2b(e^{\alpha t}\bz_t,\bu^\infty,e^{\alpha t}\bz_t)&\leq C\norm{e^{\alpha t}\bz_t}\norm{\nabla\bu^\infty}\norm{e^{\alpha t}\nabla\bz_t}\\
 &\leq \frac{\nu}{6}\norm{e^{\alpha t}\nabla\bz_t}^2+C\norm{e^{\alpha t}\bz_t}^2.
\end{align*}
Hence, we arrive at from \eqref{eq2.32}
\begin{align*}
\frac{d}{dt}(\norm{e^{\alpha t}\bz_t}^2+&\kappa\norm{e^{\alpha t}\nabla\bz_t}^2)+\nu\norm{e^{\alpha t}\nabla\bz_t}^2\leq \frac{6}{\nu}\norm{e^{\alpha t}\bF_t}^2_{-1}+C\norm{\nabla\bz}^2\norm{e^{\alpha t}\bz_t}^2+C\norm{e^{\alpha t}\bz_t}^2.
\end{align*}
Multiply the above inequality by $\tau^\beta(t)$ to obtain
\begin{align*}
 \frac{d}{dt}\Big(\tau^\beta(t)(\norm{e^{\alpha t}\bz_t}^2&+\kappa\norm{e^{\alpha t}\nabla\bz_t}^2)\Big)+\nu\tau^\beta(t)\norm{e^{\alpha t}\nabla\bz_t}^2\\
 &\leq \frac{6}{\nu}\tau^\beta(t)\norm{e^{\alpha t}\bF_t}^2_{-1}+C\tau^\beta(t)\norm{\nabla\bz(t)}^2\norm{e^{\alpha t}\bz_t}^2+C\tau^\beta(t)\norm{e^{\alpha t}\bz_t}^2\\
 &\quad+\frac{d}{dt}\big(\tau^\beta(t)\big)\Big(\frac{1}{\lambda_1}+\kappa\Big)\norm{e^{\alpha t}\nabla\bz_t}^2\\
 &\leq \frac{6}{\nu}\tau^\beta(t)\norm{e^{\alpha t}\bF_t}^2_{-1}+C\tau^\beta(t)\norm{\nabla\bz(t)}^2\norm{e^{\alpha t}\bz_t}^2+C\tau^\beta(t)\norm{e^{\alpha t}\bz_t}^2\\
 &\quad+\frac{\nu}{2}\tau^\beta(t)\norm{e^{\alpha t}\nabla\bz_t}^2.
\end{align*}
Rearrange the above inequality to 
integrate from $0$ to $t$ and then multiply the resulting inequality by $e^{-2\delta_0 t}$ to arrive at
\begin{align}
\tau^\beta(t)(\norm{e^{\alpha_1 t}\bz_t}^2+&\kappa\norm{e^{\alpha_1 t}\nabla\bz_t}^2)+\frac{\nu}{2} e^{-2\delta_0 t}\int_{0}^{t}\tau^\beta(s)\norm{e^{\alpha s}\nabla\bz_t(s)}^2 ds\notag\\
&\leq e^{-2\delta_0 t}\tau^\beta(0)(\norm{\bz_t(0)}^2+\norm{\nabla\bz_t(0)}^2)+\frac{6}{\nu}e^{-2\delta_0 t}\int_{0}^{t}\tau^\beta(s)\norm{e^{\alpha s}\bF_t}^2_{-1} ds\notag\\
&\quad+Ce^{-2\delta_0 t}\int_{0}^{t}\norm{\nabla\bz(s)}^2\norm{e^{\alpha s}\bz_t(s)}^2 ds+Ce^{-2\delta_0 t}\int_{0}^{t}\tau^\beta(s)\norm{e^{\alpha s}\bz_t(s)}^2 ds.\label{eq2.33}
\end{align}
Now from equation \eqref{eq1.13} after putting $\bphi=\bz_t$, it follows that
\begin{equation}\label{eq2.34}
 \norm{\bz_t}^2+\kappa\norm{\nabla\bz_t}^2\leq C(\nu,\lambda_1,\norm{f^\infty}_{-1})(\norm{\bF}^2+\norm{\tilde\Delta\bz}^2+\norm{\nabla\bz}^2\norm{\tilde\Delta\bz}^2).
\end{equation}
From \eqref{eq2.34}, we can find the estimate at $t=0$ that is $\norm{\bz_t(0)}^2+\kappa\norm{\nabla\bz_t(0)}^2$.
Use previous Lemma \ref{lm3} in \eqref{eq2.33} to obtain
\begin{align}
\tau^\beta(t)(\norm{e^{\alpha_1 t}\bz_t}^2+&\kappa\norm{e^{\alpha_1 t}\nabla\bz_t}^2)+\nu e^{-2\delta_0 t}\int_{0}^{t}\tau^\beta(s)\norm{e^{\alpha s}\nabla\bz_t(s)}^2 ds\notag\\
&\leq C e^{-2\delta_0 t}\tau^\beta(0)(\norm{\bF_0}^2+(1+\kappa)\norm{\tilde\Delta\bz_0}^2)+Ce^{-2\delta_0 t}\int_{0}^{t}\tau^\beta(s) \norm{e^{\alpha s}\bF_t(s)}^2_{-1} ds\notag\\
&+Ce^{-2\delta_0 t}\int_{0}^{t}\norm{\nabla\bz(s)}^2\norm{e^{\alpha s}\bz_t(s)}^2 ds+Ce^{-2\delta_0 t}\int_{0}^{t}\tau^\beta(s)\norm{e^{\alpha s}\bz_t(s)}^2 ds\label{eqx3.9}.
\end{align}
This completes the first part of the proof.
Use L'Hospital's rule to obtain from \eqref{eqx3.9}
\begin{align*}
 \limsup_{t\to\infty} \tau^\beta(t)(\norm{e^{\alpha_1 t}\bz_t}^2+\kappa\norm{e^{\alpha_1 t}\nabla\bz_t}^2)\leq C(M_1)\frac{1}{2\delta_0},
\end{align*}
and
\begin{equation*}
 \limsup_{t\to\infty} \nu e^{-2\delta_0 t}\int_{0}^{t}\tau^\beta(s)e^{2\alpha s}\norm{\nabla\bz_t(s)}^2 ds\leq C(M_1)\frac{1}{2\delta_0},
\end{equation*}
and hence,
\begin{equation*}
\limsup_{t\to\infty} \nu \tau^\beta(t)\norm{e^{\alpha_1 t}\nabla\bz_t(t)}^2\leq C(M_1).
\end{equation*}
This completes the rest of the proof.
\end{proof}
\begin{lemma}\label{lmx5}
Under the assumption $\bf{(A1)}$, let
$\limsup_{t\to\infty}\tau^\beta(t)e^{2\alpha_1 t}(\norm{\bF(t)}^2+\norm{\bF_t}^2_{-1})\leq M_1$
and
 $\limsup_{t\to\infty}\tau^\beta(t)e^{2\alpha_1 t}(\norm{\bF(t)}^2+\norm{\bF_t}^2)\leq M_2$
and $\bz_0\in {\bf H}^2\cap {\bf H}^1_0$. Then, there exists a positive 
 constant $C=C(N,\nu,\lambda_1,\norm{f^\infty}_{-1})$ such that for $\alpha_1=\alpha-\delta_0,\hspace{0.1cm} \delta_0>0$
 \begin{align*}
 \nu\tau^\beta(t)\norm{e^{\alpha_1 t}\tilde \Delta\bz}^2&+4e^{-2\delta_0 t}\int_{0}^{t}\tau^\beta(s)\big(\norm{e^{\alpha s}\nabla\bz_t(s)}^2
 +\kappa\norm{e^{\alpha s}\tilde \Delta\bz_t(s)}^2\big) ds\\
 &\leq C\tau^\beta(t)e^{2\alpha_1 t}\Big(\norm{\bF}^2+\norm{\bz}^2\norm{\nabla\bz}^4+\norm{\nabla\bz}^2\Big)\\
 &\qquad+C\tau^\beta(0)e^{-2\delta_0 t}(\norm{\bF_0}^2+\norm{\bz_0}^2_2+\norm{\bz_0}^2\norm{\nabla\bz_0}^4)\\
 &\quad+C e^{-2\delta_0 t}\int_{0}^{t}\tau^\beta(s)\big(\norm{\bz(s)}^2\norm{\nabla\bz(s)}^2\norm{e^{\alpha s}\nabla\bz(s)}^2
 +\norm{e^{\alpha s}\nabla\bz_t(s)}^2\norm{\nabla\bz}^2\big)\; ds\\
 &\quad +C e^{-2\delta_0 t}\int_{0}^{t}\tau^\beta(s)\big ( \|e^{\alpha s}\bF(s)\|^2 + \|e^{\alpha s}\bF_t(s)\|^2)\;ds.
 \end{align*}
Moreover,
\begin{align}
\limsup_{t\to\infty} \nu\tau^\beta(t)\norm{e^{\alpha_1 t}\tilde{\Delta}\bz}^2\leq \frac{C}{2\delta_0}\Big(M_1+M_2\Big)\label{eqx5.11},
\end{align}
and
\begin{align}
\limsup_{t\to\infty}\tau^\beta(t)\Big(\norm{e^{\alpha_1 t}\nabla\bz_t}^2+\kappa\norm{e^{\alpha_1 t}\tilde{\Delta}\bz_t}^2\Big)\leq C(M_1+M_2)\label{eqxx5.11}.
\end{align}
\end{lemma}
\begin{proof}
Rewrite \eqref{eq1.13} as
\begin{align}
\bz_t-\kappa\tilde{\Delta}\bz_t-\nu\tilde{\Delta}\bz+\bz\cdot\nabla\bz+\bu^\infty\cdot\nabla \bz+\bz\cdot\nabla \bu^\infty=\bF \label{eqx2.8}.
\end{align}
Form the $L^2$- inner product between \eqref{eqx2.8} and $-e^{2\alpha t}\tilde{\Delta}\bz_t$ to obtain 
\begin{align*}
\frac{\nu}{2}\frac{d}{dt}\big(\norm{e^{\alpha t}\tilde \Delta\bz}^2\big)+\big(\norm{e^{\alpha t}\nabla\bz_t}^2+\kappa\norm{e^{\alpha t}\tilde \Delta \bz_t}^2\big)&=e^{2\alpha t}(\bF,-\tilde \Delta\bz_t)+e^{2\alpha t}(\bz\cdot\nabla\bz,\tilde{\Delta }\bz_t)\\
&\quad+e^{2\alpha t}\big(\bu^\infty\cdot\nabla\bz+\bz\cdot\bu^\infty,\tilde \Delta\bz_t\big)+\nu\alpha \norm{e^{\alpha t}\tilde \Delta\bz}^2\\
&=\big(I_1+I_2+I_3\big)+\nu\alpha \norm{e^{\alpha t}\tilde \Delta\bz}^2.
\end{align*}
Multiply the above inequality by $2\tau^\beta(t)$ to arrive at
\begin{align}
\frac{d}{dt}\Big(\nu\tau^\beta(t)\norm{e^{\alpha t}\tilde \Delta\bz}^2\Big)&+2\tau^\beta(t)\big(\norm{e^{\alpha t}\nabla\bz_t}^2+\kappa\norm{e^{\alpha t}\tilde \Delta \bz_t}^2\big)\notag\\
&=2\tau^\beta(t)e^{2\alpha t}(\bF,-\tilde \Delta\bz_t)+2\tau^\beta(t)e^{2\alpha t}(\bz\cdot\nabla\bz,\tilde{\Delta }\bz_t)\notag\\
&\quad+2\tau^\beta(t)e^{2\alpha t}\big(\bu^\infty\cdot\nabla\bz+\bz\cdot\bu^\infty,\tilde \Delta\bz_t\big)+2\nu\big(\alpha\tau^\beta(t)+\frac{d}{dt}\tau^\beta(t)\big)\norm{e^{\alpha t}\tilde \Delta\bz}^2\notag\\
&\leq \big(I_1+I_2+I_3\big)+C\tau^\beta(t)\norm{e^{\alpha t}\tilde \Delta\bz}^2\label{eqx5.9}.
\end{align}
Integrate \eqref{eqx5.9} with respect to time from $0$ to $t$ and then, multiply the resulting equality by $e^{-2\delta_0 t}$ to obtain 
\begin{align}
\nu\tau^\beta(t)&\norm{e^{\alpha_1 t}\tilde \Delta\bz}^2+2e^{-2\delta_0 t}\int_{0}^{t}\tau^\beta(s)\big(\norm{e^{\alpha s}\nabla\bz_t(s)}^2+\kappa\norm{e^{\alpha s}\tilde \Delta\bz_t(s)}^2\big)\; ds\notag\\
&\leq \nu\tau^\beta(0)e^{-2\delta_0 t}\norm{\tilde \Delta\bz_0}^2+e^{-2\delta_0 t}\int_{0}^{t}\big(I_1+I_2+I_3\big)\; ds+C e^{-2\delta_0 t}\int_{0}^{t}\tau^\beta(s)\norm{e^{\alpha s}\tilde \Delta \bz}^2\; ds\label{eqx5.10}.
\end{align}
Since
\begin{align*}
I_1=2\tau^\beta(t) e^{2\alpha t}(F,-\tilde \Delta\bz_t)&=2\frac{d}{dt}\big(\tau^\beta(t)e^{2\alpha t}(\bF, -\tilde \Delta\bz)\big)
+4\alpha\tau^\beta(t) e^{2\alpha t}(\bF,\tilde \Delta\bz)+2\tau^\beta(t) e^{2\alpha t}(\bF_t,\tilde{\Delta}\bz)\\
&\qquad +2\frac{d}{dt}\big(\tau^\beta(t)\big)e^{2\alpha t}(F,\tilde \Delta\bz),
\end{align*}
therefore, it follows that
\begin{align*}
e^{-2\delta_0 t}\int_{0}^{t}I_1 ds&\leq 2\tau^\beta(t)e^{2\alpha_1 t}(\bF,-\tilde{\Delta}\bz)+2\tau^\beta(0)e^{-2\delta_0 t}(\bF_0,\tilde \Delta\bz_0)\\
&\quad+Ce^{-2\delta_0 t}\int_{0}^{t}\tau^\beta(s) e^{2\alpha s}\big((\bF,\tilde{\Delta}\bz)+(\bF_t,\tilde{\Delta}\bz)\big)\; ds\\
&\leq \frac{\nu}{6}\tau^\beta(t)\norm{e^{\alpha_1 t}\tilde\Delta\bz}^2+C\tau^\beta(t)\norm{e^{\alpha_1 t}\bF}^2+C\tau^\beta(0) e^{-2\delta_0 t}(\norm{\bF_0}^2
+\norm{\tilde \Delta\bz_0}^2)\\
&\quad+Ce^{-2\delta_0 t}\int_{0}^{t}\tau^\beta(s) \; (\norm{e^{\alpha s}\bF}^2+\norm{e^{\alpha s}\bF_t(s)}^2)+Ce^{-2\delta_0 t}\int_{0}^{t}\tau^\beta(s) \norm{e^{\alpha s}\tilde\Delta\bz}^2\; ds.
\end{align*}
The term $I_2$ is rewritten as
\begin{align*}
I_2&=2\tau^\beta(t)e^{2\alpha t}(\bz\cdot\nabla\bz,\tilde{\Delta}\bz_t)\\&=2\frac{d}{dt}\big(\tau^\beta(t)e^{2\alpha t}(\bz\cdot\nabla\bz,\tilde{\Delta}\bz)\big)-4\alpha\tau^\beta(t) e^{2\alpha t}(\bz\cdot\nabla\bz,\tilde{\Delta}\bz)-2\tau^\beta(t)e^{2\alpha t}(\bz_t\cdot\nabla\bz,\tilde{\Delta}\bz)\\
&\quad-2\tau^\beta(t)e^{2\alpha t}(\bz\cdot\nabla\bz_t,\tilde{\Delta}\bz)-2\frac{d}{dt}\big(\tau^\beta(t)\big)e^{2\alpha t}(\bz\cdot\nabla\bz,\tilde{\Delta}\bz).
\end{align*}
Hence, after bounding all the trilinear terms suitably, it follows that
\begin{align*}
e^{-2\delta_0 t}\int_{0}^{t}I_2(s) ds&\leq \frac{\nu}{6}\tau^\beta(t)\norm{e^{\alpha_1 t}\tilde{\Delta}\bz}^2+C\tau^\beta(t)\norm{\bz}^2\norm{\nabla\bz}^2\norm{e^{\alpha_1 t}\nabla\bz}^2+Ce^{-2\delta_0 t}\tau^\beta(0)\Big(\norm{\tilde{\Delta}\bz_0}^2\\
&\quad+\norm{\bz_0}^2\norm{\nabla\bz_0}^4\Big)+Ce^{-2\delta_0 t}\int_{0}^{t}\tau^\beta(s)\Big(\norm{e^{\alpha s}\tilde{\Delta}\bz}^2+\norm{\bz}^2\norm{\nabla\bz}^2\norm{e^{\alpha s}\nabla\bz}^2\\
&\qquad+\norm{e^{\alpha s}\nabla\bz_t(s)}^2\norm{\nabla\bz}^2\Big)\; ds.
\end{align*}
Similarly, 
\begin{align*}
e^{-2\delta_0 t}\int_{0}^{t}I_3(s) ds&=2\tau^\beta(t)e^{2\alpha_1 t}(\bu^\infty\cdot\nabla\bz+\bz\cdot\nabla\bu^\infty,\tilde{\Delta}\bz)-2\tau^\beta(0)e^{-2\delta_0 t}(\bu^\infty\cdot\nabla\bz_0+\bz_0\cdot\nabla\bu^\infty,\tilde{\Delta}\bz_0)\\
&\quad-4\alpha e^{-2\delta_0 t}\int_{0}^{t}\tau^\beta(s) e^{2\alpha s}(\bu^\infty\cdot\nabla\bz+\bz\cdot\nabla\bu^\infty,\tilde{\Delta}\bz)\; ds\\
&\quad-2e^{-2\delta_0 t}\int_{0}^{t}\tau^\beta(s)e^{2\alpha s}(\bu^\infty\cdot\nabla\bz_t+\bz_t\cdot\nabla\bu^\infty,\tilde{\Delta}\bz)\; ds\\
&\leq \frac{\nu}{6}\tau^\beta(t)\norm{e^{\alpha_1 t}\tilde{\Delta}\bz}^2+C\tau^\beta(t)\norm{\nabla\bu^\infty}^2\norm{e^{\alpha_1 t}\nabla\bz}^2+C\tau^\beta(0)e^{-2\delta_0 t}\norm{\bz_0}^2_2\\
&\quad+Ce^{-2\alpha t}\int_{0}^{t}\tau^\beta(s) \Big(\norm{e^{\alpha s}\nabla\bz}^2+\norm{e^{\alpha s}\nabla\bz_t(s)}^2 
\Big)\; ds.
\end{align*}
Hence from \eqref{eqx5.10}, we arrive at
\begin{align*}
\nu\tau^\beta(t)\norm{e^{\alpha_1 t}\tilde \Delta\bz}^2&+4e^{-2\delta_0 t}\int_{0}^{t}\tau^\beta(s)\big(\norm{e^{\alpha s}\nabla\bz_t(s)}^2
+\kappa\norm{e^{\alpha s}\tilde \Delta\bz_t(s)}^2\big) ds\\
&\leq C\tau^\beta(t)\Big(\norm{e^{\alpha_1 t}\bF}^2+\norm{\bz}^2\norm{\nabla\bz}^2\norm{e^{\alpha_1 t}\nabla\bz}^2+\norm{e^{\alpha_1 t}\nabla\bz}^2\Big)+Ce^{-2\delta_0 t}\tau^\beta(0)\Big(\norm{\bF_0}^2&\\
&\quad+\norm{\bz_0}^2_2+\norm{\bz_0}^2\norm{\nabla\bz_0}^4\Big)
+C e^{-2\delta_0 t}\int_{0}^{t}\tau^\beta(s)\Big(\norm{e^{\alpha s}\tilde{\Delta}\bz}^2
+\norm{e^{\alpha s}\nabla\bz}^2\\
&\quad +\norm{e^{\alpha s}\nabla\bz_t}^2+\norm{\bz}^2\norm{\nabla\bz}^2\norm{e^{\alpha s}\nabla\bz}^2
+\norm{e^{\alpha s}\nabla\bz_t(s)}^2\norm{\nabla\bz}^2\Big)\; ds\\
&\qquad+C e^{-2\delta_0 t}\int_{0}^{t}\tau^\beta(s)\big ( \|e^{\alpha s}\bF(s)\|^2 + \|e^{\alpha s}\bF_t(s)\|^2)\;ds.
\end{align*}
This completes the first part of the proof.\\
Again use L'Hospital's rule and approach as in previous Lemmas \ref{lm1}-\ref{lm4} to obtain
\begin{align*}
\limsup_{t\to\infty} \nu\tau^\beta(t)\norm{e^{\alpha_1 t}\tilde{\Delta}\bz}^2\leq \frac{C}{2\delta_0}\Big(M_1+M_2\Big),
\end{align*}
and
\begin{align*}
\limsup_{t\to\infty}\tau^\beta(t)\Big(\norm{e^{\alpha_1 t}\nabla\bz_t}^2+\kappa\norm{e^{\alpha_1 t}\tilde{\Delta}\bz_t}^2\Big)\leq C(M_1+M_2).
\end{align*}
Now 
\begin{align*}
 \kappa\norm{\tilde\Delta\bz_t}\leq \norm{\bz_t}+\nu\norm{\tilde\Delta\bz}+C\norm{\bz}^\frac{1}{2}\norm{\nabla\bz}\norm{\tilde\Delta\bz}^\frac{1}{2}+C\norm{\nabla\bz}+\norm{\bF}.
\end{align*}
Hence, it follows that
\begin{align*}
\norm{\kappa\tilde{\Delta}\bz_t}^2\leq C\Big(\norm{\bz_t}^2+\norm{\tilde{\Delta}\bz}^2+\norm{\bz}^2\norm{\nabla\bz}^4
+\norm{\nabla\bz}^2+\norm{\bF}^2\Big).
\end{align*}
A use of previous Lemmas and \eqref{eqx5.11} yields
\begin{align}
\limsup_{t\to\infty}\tau^\beta(t)\norm{e^{\alpha_1 t}\kappa\tilde{\Delta}\bz_t}^2\leq
C(M_1+M_2)\label{eqx5.12}.
\end{align}
This completes the rest of the proof.
\end{proof}
\begin{lemma}\label{lm5}
 Under the assumption $\bf{(A1)}$, let $\tau^\beta(t)e^{2\alpha_1 t}\norm{\bF(t)}^2\leq M$ $\forall t\geq 0$. Then, there exists a positive 
 constant $C=C(N,\nu,\lambda_1,\norm{f^\infty}_{-1})$ such that
 \begin{equation}\label{eqx5.1}
 \tau^\beta(t)\norm{e^{\alpha_1 t}\bz(t)}^2_{\bf {H}^1}\leq CM \quad \forall t,\beta\geq 0, 
 \end{equation}
 holds.
\end{lemma}
\begin{proof}
 From Lemma \ref{lm1}, it follows that
 \begin{align}
  \tau^\beta(t)(\norm{e^{\alpha_1 t}\bz(t)}^2+\kappa\norm{e^{\alpha_1 t}\nabla\bz(t)}^2)&\leq \tau^\beta(0)e^{-2\delta_0 t}(\norm{\bz_0}^2+\kappa\norm{\nabla\bz_0}^2)+\frac{2}{\gamma_1\lambda_1}e^{-2\delta_0 t}\int_{0}^{t}\tau^\beta(s)e^{2\alpha s}\norm{\bF}^2 ds\notag\\
  &\leq \tau^\beta(0)e^{-2\delta_0 t}(\norm{\bz_0}^2+\kappa\norm{\nabla\bz_0}^2)+\frac{2}{\gamma_1\lambda_1}M\frac{1-e^{-2\delta_0 t}}{2\delta_0}\notag\\
  &\leq \tau^\beta(0)(\norm{\bz_0}^2+\kappa\norm{\nabla\bz_0}^2)+\frac{1}{\gamma_1\lambda_1\delta_0}M\label{eq5.2}.
 \end{align}
Also from Lemma \ref{lm1}, we find that
\begin{equation}\label{eq5.3}
\frac{d}{dt}(\norm{e^{\alpha_1 t}\bz(t)}^2+\kappa\norm{e^{\alpha_1 t}\nabla\bz(t)}^2)+\frac{\gamma_1}{2}\norm{e^{\alpha_1 t}\nabla\bz}^2\leq \frac{2}{\gamma_1\lambda_1}\norm{e^{\alpha_1 t}\bF}^2.
\end{equation}
Integrate \eqref{eq5.3} from $0$ to $t$ to obtain
\begin{align*}
 \frac{\gamma_1}{2}\int_{0}^{t}e^{2\alpha_1 s}\norm{\nabla\bz(s)}^2 ds&\leq (\norm{\bz_0}^2+\kappa\norm{\nabla\bz_0}^2)+\frac{2}{\gamma_1\lambda_1}\int_{0}^{t}e^{2\alpha_1 s}\norm{\bF}^2 ds\\
 &\leq (\norm{\bz_0}^2+\kappa\norm{\nabla\bz_0}^2)+\frac{2}{\gamma_1\lambda_1}Mt.
\end{align*}
Therefore, we arrive at
\begin{equation}\label{eq5.4}
 \int_{0}^{t}\norm{\bz(s)}^2\norm{\nabla\bz(s)}^2 ds\leq C(1+t)\quad \forall t\geq 0.
\end{equation}
Now from \eqref{eqx3.33} in Lemma \ref{lm2}, it follows that by using Gronwall's inequality
\begin{align}
\tau^\beta(t)(\norm{e^{\alpha_1 t}\nabla\bz}^2&+\kappa\norm{e^{\alpha_1 t}\tilde\Delta\bz}^2)\notag\\
 &\leq \Big(e^{-2\delta_0 t}\tau^\beta(0)\{\norm{\nabla\bz_0}^2+\kappa\norm{\tilde\Delta\bz_0}^2\}+C(\nu)e^{-2\delta_0 t}\int_{0}^{t}\tau^\beta(s)e^{2\alpha s}\norm{\bF}^2 ds\Big)\notag\\
 &\times\exp\{C(N,\nu,\lambda_1,\norm{f^\infty}_{-1})\int_{0}^{t}\norm{\bz(s)}^2\norm{\nabla\bz(s)}^2 ds\}\notag\\
 &\leq \Big(e^{-2\delta_0 t}\{\norm{\nabla\bz_0}^2+\kappa\norm{\tilde\Delta\bz_0}^2\}+C(\nu)M\frac{(1-e^{-2\delta_0 t})}{2\delta_0}\Big)\notag\\
  &\times\exp\{C(M)(1+t)\}\label{eq5.5}.
\end{align}
Now, \eqref{eq5.5} holds for finite $t, 0<t\leq T$, where $0<T<\infty$. For large $t>T$, we note from
 Lemma \ref{lm1}, that
\begin{align}
 \limsup_{t\to\infty}\tau^\beta(t)\norm{e^{\alpha_1 t}\nabla\bz}^2\leq C\limsup_{t\to\infty}\tau^\beta(t)\norm{e^{\alpha_1 t}\bF}^2\leq CM\label{eq5.6}.
\end{align}
Therefore from \eqref{eq5.6} and  for a finite time $T>0$ there holds
\begin{equation}\label{eq5.7}
 \tau^\beta(t)\norm{e^{\alpha_1 t}\nabla\bz(t)}^2\leq CM\quad \forall t\geq T.
\end{equation}
Hence, \eqref{eq5.7} with \eqref{eq5.5} and \eqref{eq5.2} complete the proof.
\end{proof}
\begin{lemma}\label{lm6}
Under the assumption $\bf{(A1)}$, let $\tau^\beta(t)e^{2\alpha_1 t}(\norm{\bF(t)}^2+\norm{\bF_t}^2_{-1})\leq M_1$  and 

$\tau^\beta(t)e^{2\alpha_1 t}(\norm{\bF(t)}^2+\norm{\bF_t}^2)\leq M_2$ $\forall t\geq 0$.
 Then, there exists a positive 
 constant $C=C(N,\nu,\lambda_1,\norm{f^\infty}_{-1})$ such that for $t\geq 0$
\begin{equation*}
 \tau^\beta(t)(\norm{e^{\alpha_1 t}\bz_t}^2+\kappa\norm{e^{\alpha_1 t}\nabla\bz_t}^2)\leq C \tau^\beta(0)(\norm{\bF_0}^2+\norm{\bz_0}^2+\kappa\norm{\nabla\bz_0}^2+\kappa\norm{\tilde\Delta\bz_0}^2)+C\frac{(M_1)}{2\delta_0},
\end{equation*}

 \begin{equation*}
 \nu\tau^\beta(t)\norm{e^{\alpha_1 t}\tilde\Delta\bz}^2\leq C\tau^\beta(0)e^{-2\delta_0 t}(\norm{\bz_0}^2+\kappa\norm{\nabla\bz_0}^2+\kappa\norm{\tilde\Delta\bz_0}^2)+C(M_1+M_2),
\end{equation*}
and
 \begin{equation*}
\tau^\beta(t)\norm{e^{\alpha_1 t}\kappa\tilde\Delta\bz_t}^2\leq  C\tau^\beta(0)e^{-2\delta_0 t}(\norm{\bz_0}^2+\kappa\norm{\nabla\bz_0}^2+\kappa\norm{\tilde\Delta\bz_0}^2)+C(M_1+M_2),
\end{equation*}
hold.
\end{lemma}
\begin{proof}
From Lemma \ref{lm4}, we obtain
 \begin{align*}
 \tau^\beta(t)(\norm{e^{\alpha_1 t}\bz_t}^2&+\kappa\norm{e^{\alpha_1 t}\nabla\bz_t}^2)+\nu e^{-2\delta_0 t}\int_{0}^{t}\tau^\beta(s) \norm{e^{\alpha s}\nabla\bz_t(s)}^2 ds\\ 
 &\leq C e^{-2\delta_0 t}\tau^\beta(0)(\norm{\bF_0}^2+(1+\kappa)\norm{\tilde\Delta\bz_0}^2)+Ce^{-2\delta_0 t}\int_{0}^{t}\tau^\beta(s) \norm{e^{\alpha s}\bF_t(s)}^2_{-1} ds\\
 &\quad+Ce^{-2\delta_0 t}\int_{0}^{t}\norm{\nabla\bz(s)}^2\norm{e^{\alpha s}\bz_t(s)}^2 ds+Ce^{-2\delta_0 t}\int_{0}^{t}\norm{e^{\alpha s}\bz_t(s)}^2 ds.
 \end{align*}
 Use the Gr\"onwall's inequality and Lemma \ref{lm4} to proceed as in Lemma \ref{lm5} to arrive at
\begin{align}
 \tau^\beta(t)\big(\norm{e^{\alpha_1 t}\bz_t}^2+&\kappa\norm{e^{\alpha_1 t}\nabla\bz_t}^2\big)+\nu e^{-2\delta_0 t}\int_{0}^{t}\tau^\beta(s) \norm{e^{\alpha s}\nabla\bz_t(s)}^2 ds\notag\\
&\leq C e^{-2\delta_0 t}\tau^\beta(0)(\norm{\bF_0}^2+\norm{\bz_0}^2+\kappa\norm{\nabla\bz_0}^2+\kappa\norm{\tilde\Delta\bz_0}^2)\notag\\
&+CM_1\frac{(1-e^{-2\delta_0 t})}{2\delta_0}\notag\\
&\leq C e^{-2\delta_0 t}\tau^\beta(0)(\norm{\bF_0}^2+\norm{\bz_0}^2+\kappa\norm{\nabla\bz_0}^2+\kappa\norm{\tilde\Delta\bz_0}^2)+C\frac{(M_1)}{2\delta_0}\label{eqx5.8}.
\end{align}
From Lemma \ref{lmx5}, it follows that
\begin{align*}
 \nu\tau^\beta(t)\norm{e^{\alpha_1 t}\tilde \Delta\bz}^2&+4e^{-2\delta_0 t}\int_{0}^{t}\tau^\beta(s)e^{2\alpha s}\big(\norm{\nabla\bz_t(s)}^2
 +\kappa\norm{\tilde \Delta\bz_t(s)}^2\big) ds\\
 &\leq C\tau^\beta(t)e^{2\alpha_1 t}\Big(\norm{\bF}^2+\norm{\bz}^2\norm{\nabla\bz}^4+\norm{\nabla\bz}^2\Big)\\
 &\qquad+C\tau^\beta(0)e^{-2\delta_0 t}(\norm{\bF_0}^2+\norm{\bz_0}^2_2+\norm{\bz_0}^2\norm{\nabla\bz_0}^4)\\
 &\quad+C e^{-2\delta_0 t}\int_{0}^{t}\tau^\beta(s)e^{2\alpha s}\big(\norm{\bz(s)}^2\norm{\nabla\bz(s)}^4
 +\norm{\nabla\bz_t(s)}^2\norm{\nabla\bz}^2\big)\; ds\\
 &\quad +C e^{-2\delta_0 t}\int_{0}^{t}\tau^\beta(s)e^{2\alpha s}\big ( \|\bF(s)\|^2 + \|\bF_t(s)\|^2)\;ds.
 \end{align*}
A use of previous Lemmas \ref{lm5} and \ref{lm5} with \eqref{eqx5.8} implies
\begin{equation}\label{eq6.1}
 \nu\tau^\beta(t)\norm{e^{\alpha_1 t}\tilde\Delta\bz}^2\leq C\tau^\beta(0)e^{-2\delta_0 t}(\norm{\bz_0}^2+\kappa\norm{\nabla\bz_0}^2+\kappa\norm{\tilde\Delta\bz_0}^2)+C(M+M_1).
\end{equation}
Now 
\begin{align*}
 \kappa\norm{\tilde\Delta\bz_t}\leq \norm{\bz_t}+\nu\norm{\tilde\Delta\bz}+C\norm{\bz}^\frac{1}{2}\norm{\nabla\bz}\norm{\tilde{\Delta}\bz}+C\norm{\nabla\bz}+\norm{\bF}.
\end{align*}
Therefore, we arrive at 
\begin{align*}
 \tau^\beta(t)\norm{e^{\alpha_1 t}\kappa\tilde\Delta\bz_t}^2&\leq C\tau^\beta(t)\norm{e^{\alpha_1 t}\bz_t}^2+C\tau^\beta(t)\norm{e^{\alpha_1 t}\tilde\Delta\bz}^2+
 C\tau^\beta(t)\norm{\bz}^2\norm{\nabla\bz}^2\norm{e^{\alpha_1 t}\nabla\bz}^2\\
 &+C\tau^\beta(t)\norm{e^{\alpha_1 t}\nabla\bz}^2+C\tau^\beta(t)\norm{e^{\alpha_1 t}\bF}^2.
\end{align*}
A use of previous Lemmas and \eqref{eq6.1} yields
\begin{equation}\label{eq6.2}
\tau^\beta(t)\norm{e^{\alpha_1 t}\kappa\tilde\Delta\bz_t}^2\leq  C\tau^\beta(0)e^{-2\delta_0 t}(\norm{\bz_0}^2+\kappa\norm{\nabla\bz_0}^2+\kappa\norm{\tilde\Delta\bz_0}^2)+C(M_1+M_2).
\end{equation}
This completes the rest of the proof.
\end{proof}
\begin{lemma}\label{lm7}
Under the assumption $\bf {(A1)}$, let $\tau^\beta(t)e^{2\alpha_1 t}(\norm{\bF(t)}^2+\norm{\bF_t}^2_{-1})\leq M_1$ and 

 $\tau^\beta(t)e^{2\alpha_1 t}(\norm{\bF(t)}^2+\norm{\bF_t}^2)\leq M_2$
 $\forall t\geq 0$. Then, there exists a positive 
 constant $C=C(N,\nu,\lambda_1,\norm{f^\infty}_{-1})$ such that
\begin{equation}
 \tau^\beta(t)\norm{e^{\alpha_1 t}q(t)}^2_{H^1(\Omega)/\mathbb{R}}\leq C \tau^\beta(0)e^{-2\delta_0 t}(\norm{\bz_0}^2+\kappa\norm{\nabla\bz_0}^2+\kappa\norm{\tilde\Delta\bz_0}^2)+C(M_1+M_2).
\end{equation}
\end{lemma}
\begin{proof}
Use property of the divergence free space ${\bf J}_1$ to obtain for $\bphi\in {\bf H}^1_0$
 \begin{align*}
  (\nabla q,\bphi)=(\bz_t-\kappa\tilde\Delta\bz_t-\nu\tilde\Delta\bz+\bz\cdot\nabla\bz+\bu^\infty\cdot\nabla\bz+\bz\cdot\nabla\bu^\infty-\bF,\bphi).
 \end{align*}
Hence, it follows that
\begin{align*}
 |(\nabla q,\bphi)|&\leq \norm{\bz_t}\norm{\bphi}+\kappa\norm{\tilde\Delta\bz_t}\norm{\bphi}+\nu\norm{\tilde\Delta\bz}\norm{\bphi}+C\norm{\bz}^{1/2}\norm{\nabla\bz}\norm{\tilde\Delta\bz}^{1/2}\norm{\bphi}\\
 &+C(\nu,\lambda_1,\norm{f^\infty}_{-1})\norm{\nabla\bz}\norm{\bphi}+\norm{\bF}\norm{\phi},
\end{align*}
and
\begin{align}
\tau^\delta(t)\norm{e^{\alpha_1 t}\nabla q}\leq \frac{|(\tau^\beta(t)e^{\alpha_1 t}\nabla q,\bphi)|}{\norm{\bphi}}&\leq C\tau^\delta(t)\Big(\norm{e^{\alpha_1 t}\bz_t}+\kappa\norm{e^{\alpha_1 t}\tilde\Delta\bz_t}+\norm{e^{\alpha_1 t}\tilde\Delta\bz}\notag\\
&+e^{\alpha_1 t}\norm{\bz}^{1/2}\norm{\nabla\bz}\norm{\tilde\Delta\bz}^{1/2}+\norm{e^{\alpha_1 t}\nabla\bz}+\norm{e^{\alpha_1 t}\bF}\Big)\label{eq6.10}.
\end{align}
Hence, squaring both sides, the above inequality can be rewritten as
 \begin{align}
  \tau^\beta(t)\norm{e^{\alpha_1 t}\nabla q}^2&\leq C \Big(\tau^\beta(t)\norm{e^{\alpha_1 t}\bz_t}^2+\tau^\beta(t)\norm{\kappa e^{\alpha_1 t }\tilde\Delta\bz_t}^2+\tau^\beta(t)\norm{e^{\alpha_1 t}\tilde\Delta\bz}^2\notag\\
  &\quad+\tau^\beta(t)\norm{\bz}^2\norm{\nabla\bz}^2\norm{e^{\alpha_1 t}\nabla\bz}^2+\tau^\beta(t)\norm{e^{\alpha_1 t}\nabla\bz}^2+\tau^\beta(t)\norm{e^{\alpha_1 t}\bF}^2\Big)\label{eqx6.6}.
 \end{align}
Also from \eqref{eqx1.14}, it follows that
\begin{equation}\label{eqxx1.14}
 c\tau^\delta(t)\norm{e^{\alpha_1 t}q(t)}\leq \sup_{v\in {\bf H}^1_0}\frac{(\bv,\tau^\delta(t)e^{\alpha_1 t}\nabla q(t))}{\norm{\nabla \bv}}\leq \dfrac{1}{\sqrt{\lambda_1}}\tau^\delta(t)\norm{e^{\alpha_1 t}\nabla q(t)}.
\end{equation}
Therefore, using previous Lemmas \ref{lm5} and \ref{lm6} we obtain from \eqref{eqx6.6}
\begin{align}
\tau^\beta(t)\norm{e^{\alpha_1 t} q}^2_{H^1/\mathbb{R}}\leq \tau^\beta(0)e^{-2\delta_0 t}(\norm{\bz_0}^2+\kappa\norm{\nabla\bz_0}^2+\kappa\norm{\tilde\Delta\bz_0}^2)+C(M_1+M_2).
\end{align}
This concludes the proof.
\end{proof}
Below, we prove one of the main theorem of this paper
\begin{theorem}\label{thm1}
  Under the assumption $\bf{(A1)}$, let $\limsup_{t\to\infty}\tau^\beta(t)e^{2\alpha_1 t}\norm{\bF(t)}^2\leq M$,
  and $\bz_0\in {\bf H}^2\cap {\bf H}^1_0$. Then, there exists a positive 
 constant $C=C(N,\nu,\lambda_1,\norm{f^\infty}_{-1})$ such that
 \begin{align}
  \limsup_{t\to\infty} \tau^\beta(t)\Big(\norm{e^{\alpha_1 t}\bz(t)}^2_{{\bf H}^2}+(\norm{e^{\alpha_1 t}\bz_t(t)}^2+\kappa\norm{e^{\alpha_1 t}\nabla\bz_t(t)}^2)+\norm{e^{\alpha_1 t}q(t)}^2_{H^1(\Omega)/\mathbb{R}}\Big)\leq C(M).
 \end{align} 
\end{theorem}
\begin{proof}
From equation \eqref{eq6.10}, we obtain
 \begin{align}
  \tau^\beta(t)\norm{e^{\alpha_1 t}\nabla q}^2&\leq C \Big(\tau^\beta(t)\norm{e^{\alpha_1 t}\bz_t}^2+\tau^\beta(t)\norm{\kappa e^{\alpha_1 t }\tilde\Delta\bz_t}^2+\tau^\beta(t)\norm{e^{\alpha_1 t}\tilde\Delta\bz}^2\notag\\
  &+\tau^\beta(t)\norm{\bz}^2\norm{\nabla\bz}^2\norm{e^{\alpha_1 t}\nabla\bz}^2+\tau^\beta(t)\norm{e^{\alpha_1 t}\nabla\bz}^2+\tau^\beta(t)\norm{e^{\alpha_1 t}\bF}^2\Big)\label{eqxx6.6}.
 \end{align}
Also from \eqref{eqx1.14}, it follows that
\begin{equation*}
 c\tau^\delta(t)\norm{e^{\alpha_1 t}q(t)}\leq \sup_{\bv\in {\bf H}^1_0}\frac{(\bv,\tau^\delta(t)e^{\alpha_1 t}\nabla q(t))}{\norm{\nabla \bv}}\leq \dfrac{1}{\sqrt{\lambda_1}}\tau^\delta(t)\norm{e^{\alpha_1 t}\nabla q(t)}.
\end{equation*}
Hence, we arrive at from \eqref{eqxx6.6}
\begin{align}
 \limsup_{t\to\infty} \tau^\beta(t)\norm{e^{\alpha_1 t}q(t)}^2_{H^1/\mathbb{R}}&\leq C\Big(\limsup_{t\to\infty}\tau^\beta(t)\norm{e^{\alpha_1 t}\bz_t}^2+\limsup_{t\to\infty}\tau^\beta(t)\norm{\kappa e^{\alpha_1 t}\tilde\Delta\bz_t}^2\notag\\
 &\quad+\limsup_{t\to\infty}\tau^\beta(t)\norm{e^{\alpha_1 t}\tilde\Delta\bz}^2+\limsup_{t\to\infty}\tau^\beta(t)\norm{\bz}^2\norm{\nabla\bz}^2\norm{e^{\alpha_1 t}\nabla\bz}^2\notag\\
 &\quad+\limsup_{t\to\infty}\tau^\beta(t)\norm{e^{\alpha_1 t}\nabla\bz}^2+\limsup_{t\to\infty}\tau^\beta(t)\norm{e^{\alpha_1 t}\bF}^2\Big)\notag\\
 &\leq C(M)+\limsup_{t\to\infty}\tau^\beta(t)\norm{\kappa e^{\alpha_1 t}\tilde\Delta\bz_t}^2\label{eqxxx6.6}.
\end{align}
From the equation \eqref{eq1.13}, we obtain
\begin{align}
 \limsup_{t\to\infty}\tau^\beta(t)\norm{e^{\alpha_1 t}\kappa\tilde\Delta\bz_t}^2&\leq C\Big(\limsup_{t\to\infty}\tau^\beta(t)\norm{e^{\alpha_1 t}\bz_t}^2+\limsup_{t\to\infty}\tau^\beta(t)\norm{e^{\alpha_1 t}\tilde\Delta\bz}^2\notag\\
 &\quad+\limsup_{t\to\infty}\tau^\beta(t)\norm{\bz}^2\norm{\nabla\bz}^2\norm{e^{\alpha_1 t}\nabla\bz}^2\notag\\
 &\quad+\limsup_{t\to\infty}\tau^\beta(t)\norm{e^{\alpha_1 t}\nabla\bz}^2+\limsup_{t\to\infty}\tau^\beta(t)\norm{e^{\alpha_1 t}\bF}^2\Big)\notag\\
 &\leq CM\label{eqxxx6.7}.
\end{align}
Hence, a use of \eqref{eqxxx6.7} in \eqref{eqxxx6.6} shows
\begin{equation}\label{eq6.7}
\limsup_{t\to\infty} \tau^\beta(t)\norm{e^{\alpha_1 t}q(t)}^2_{H^1/\mathbb{R}} \leq C(M).  
\end{equation}
The rest of the proof follows from Lemmas \ref{lm1}, \ref{lm2} and \ref{lm3}.

\end{proof}
\begin{remark}
 If the forcing function satisfies the property
 \begin{equation}\label{eqx6.7}
 \limsup_{t\to\infty} t^\beta e^{2\alpha_1 t}\norm{\bf {F}(t)}^2=0, 
 \end{equation}
then as a consequence of Theorem \ref{thm1}, we obtain for $0<\bar t\leq t$
\begin{equation}\label{eqxx6.7}
 \limsup_{t\to\infty} t^\beta(t)\Big(\norm{e^{\alpha_1 t}\bz(t)}^2_{{\bf H}^2}+(\norm{e^{\alpha_1 t}\bz_t(t)}^2+\kappa\norm{e^{\alpha_1 t}\nabla\bz_t}^2)+\norm{e^{\alpha_1 t}q(t)}^2_{H^1(\Omega)/\mathbb{R}}\Big)=0,
\end{equation}
and as $t\to\infty$
\begin{equation}\label{eqx6.8}
\Big(\norm{\bz(t)}^2_{{\bf H}^2}+(\norm{\bz_t(t)}^2+\kappa\norm{\nabla\bz_t}^2)+\norm{q(t)}^2_{H^1(\Omega)/\mathbb{R}}\Big)=O(t^{-\beta}e^{-2\alpha_1 t})\quad as \quad t\to\infty.
\end{equation}
\end{remark}
Below we prove two Theorems which are valid for all $t>0$.
\begin{theorem}\label{thm2}
Under the assumption $\bf{(A1)}$, let $\tau^\beta(t)e^{2\alpha_1 t}\norm{\bF(t)}^2\leq M$,
  and $\bz_0\in {\bf H}^2\cap {\bf H}^1_0$. Then, there exists a positive 
 constant $C=C(N,\nu,\lambda_1,\norm{f^\infty}_{-1})$ such that for all $t>0$
 \begin{align*}
 \tau^\beta(t)\norm{e^{\alpha_1 t}\bz(t)}^2_{\bf {H}^1}&+e^{-2\delta_0 t}\int_{0}^{t}\tau^\beta(s)(\norm{e^{\alpha s}\bz_t(s)}^2+\kappa\norm{e^{\alpha s}\nabla\bz_t(s)}^2) ds\\
 &+e^{-2\delta_0 t}\int_{0}^{t}\tau^\beta(s)\norm{e^{\alpha s}\tilde\Delta\bz(s)} ^2 ds+e^{-2\delta_0 t}\int_{0}^{t}\tau^\beta(s)e^{2\alpha s}\norm{ q}^2_{H^1(\Omega)/\mathbb{R}} ds\\
 &\leq C\frac{M}{2\delta_0}+C\tau^\beta(0)e^{-2\delta_0 t}\big(\norm{\bz_0}^2+\kappa\norm{\nabla\bz_0}^2+\kappa\norm{\tilde\Delta\bz_0}^2\big).
 \end{align*}
\end{theorem}
\begin{proof}
 From Lemmas \ref{lm2} and \ref{lm3}, we find that
 \begin{align*}
  e^{-2\delta_0 t}\Big(\int_{0}^{t}\tau^\beta(s)&(\norm{e^{\alpha s}\bz_t(s)}^2+\kappa\norm{e^{\alpha s}\nabla\bz_t(s)}^2) ds+\int_{0}^{t}\tau^\beta(s)\norm{e^{\alpha s}\tilde\Delta\bz(s)} ^2 ds\Big)\\
  &\leq C e^{-2\delta_0 t}\Big(\int_{0}^{t}\tau^\beta(s)e^{2\alpha s}\norm{\bF}^2 ds+\tau^\beta(0)(\norm{\bz_0}^2+\kappa\norm{\nabla\bz_0}^2+\kappa\norm{\tilde\Delta\bz_0}^2)\\
  &+\int_{0}^{t}\tau^\beta(s)\norm{\bz(s)}^2\norm{\nabla\bz(s)}^2\norm{e^{\alpha s}\nabla\bz(s)}^2 ds\Big)\\
  &\leq CM\frac{1-e^{-2\delta_0 t}}{2\delta_0}+e^{-2\delta_0 t}\tau^\beta(0)\big(\norm{\bz_0}^2+\kappa\norm{\nabla\bz_0}^2+\kappa\norm{\tilde\Delta\bz_0}^2\big)\\
  &\quad+e^{-2\delta_0 t}\int_{0}^{t}\tau^\beta(s)\norm{e^{\alpha s}\nabla\bz(s)}^2 ds\Big)\\
  &\leq CM\frac{1-e^{-2\delta_0 t}}{2\delta_0}+e^{-2\delta_0 t}\tau^\beta(0)(\norm{\bz_0}^2+\kappa\norm{\nabla\bz_0}^2+\kappa\norm{\tilde\Delta\bz_0}^2)\Big)\\
  &\leq C\frac{M}{2\delta_0}+C\tau^\beta(0)e^{-2\delta_0 t}\big(\norm{\bz_0}^2+\kappa\norm{\nabla\bz_0}^2+\kappa\norm{\tilde\Delta\bz_0}^2\big).
 \end{align*}
Also, it follows that
 \begin{align*}
 e^{-2\delta_0 t}\int_{0}^{t}\tau^\beta(s) & e^{2\alpha s}\norm{\kappa\tilde\Delta\bz_t(s)}^2 ds\\
 &\leq C\Big(e^{-2\delta_0 t}\int_{0}^{t}\tau^\beta(s)\norm{e^{\alpha s}\bz_t(s)}^2 ds+Ce^{-2\delta_0 t}\int_{0}^{t}\tau^\beta(s)\norm{e^{\alpha s}\tilde\Delta\bz(s)}^2 ds\\
 &+e^{-2\delta_0 t}\int_{0}^{t}\tau^\beta(s)\norm{\nabla\bz(s)}^2\norm{e^{\alpha s}\nabla\bz(s)}^2 ds\\
 &+e^{-2\delta_0 t}\int_{0}^{t}\tau^\beta(s)\norm{e^{\alpha s}\nabla\bz(s)}^2 ds+e^{-2\delta_0 t}\int_{0}^{t}\tau^\beta(s)\norm{e^{\alpha s}\bF(s)}^2 ds\Big)\\
 &\leq C\frac{M}{2\delta_0}+C\tau^\beta(0)e^{-2\delta_0 t}\big(\norm{\bz_0}^2+\kappa\norm{\nabla\bz_0}^2+\kappa\norm{\tilde\Delta\bz_0}^2\big).
\end{align*} 
Now from equations \eqref{eqx6.6} and \eqref{eqxx1.14} in Lemma \ref{lm7}, we obtain 
\begin{align*}
 e^{-2\delta_0 t}\int_{0}^{t}\tau^\beta(s) & e^{2\alpha s}\norm{ q}^2_{H^1(\Omega)/\mathbb{R}} ds\\
  &\leq C\Big(e^{-2\delta_0 t}\int_{0}^{t}\tau^\beta(s)\norm{e^{\alpha s}\bz_t(s)}^2 ds +e^{-2\delta_0 t}\int_{0}^{t}\tau^\beta(s)\norm{e^{\alpha s}\kappa\tilde\Delta\bz_t(s)}^2 ds\\
 &\quad+e^{-2\delta_0 t}\int_{0}^{t}\tau^\beta(s)\norm{e^{\alpha s}\tilde\Delta\bz(s)}^2 ds++e^{-2\delta_0 t}\int_{0}^{t}\tau^\beta(s)\norm{e^{\alpha s}\nabla\bz(s)}^2 ds\\
 &\quad+e^{-2\delta_0 t}\int_{0}^{t}\tau^\beta(s)\norm{\bz(s)}^2\norm{\nabla\bz(s)}^2\norm{e^{\alpha s}\nabla\bz(s)}^2 ds\\
 &\quad+e^{-2\delta_0 t}\int_{0}^{t}\tau^\beta(s)\norm{e^{\alpha s}\bF}^2 ds\Big)\\
 &\leq C\frac{M}{2\delta_0}+C\tau^\beta(0)e^{-2\delta_0 t}(\norm{\bz_0}^2+\kappa\norm{\nabla\bz_0}^2+\kappa\norm{\tilde\Delta\bz_0}^2).
\end{align*}
The rest of the proof follows from the Lemma \ref{lm5}.
\end{proof}
\begin{theorem}
Under the assumption $\bf{(A1)}$, let $\tau^\beta(t)e^{2\alpha_1 t}(\norm{\bF(t)}^2+\norm{\bF_t}^2_{-1})\leq M_1$ 
and 

$\tau^\beta(t)e^{2\alpha_1 t}(\norm{\bF(t)}^2+\norm{\bF_t}^2)\leq M_2$
$\forall t\geq 0$. Then, there exists a positive 
 constant $C=C(N,\nu,\lambda_1,\norm{f^\infty}_{-1})$ such that for all $t>0$
 \begin{align*}
  \tau^\beta(t)\Big(\norm{e^{\alpha_1 t}\bz(t)}^2_{{\bf H}^2}&+\norm{e^{\alpha_1 t}\bz_t(t)}^2+\norm{e^{\alpha_1 t}q(t)}^2_{H^1(\Omega)/\mathbb{R}}\Big)\\
  &\leq \tau^\beta(0)e^{-2\delta_0 t}((1+\kappa)\norm{\tilde\Delta\bz_0}^2)+C(M_1+M_2),
 \end{align*}
and
\begin{align*}
 \tau^\beta(t)\big(\kappa\norm{e^{\alpha_1 t}\nabla\bz_t}^2+\norm{e^{\alpha_1 t}\kappa\tilde\Delta\bz_t}^2\big)\leq \tau^\beta(0)e^{-2\delta_0 t}(1+\kappa)\norm{\tilde\Delta\bz_0}^2+C(M_1+M_2).
\end{align*}
hold.
\end{theorem}
\begin{proof}
 Proof follows from Lemmas \ref{lm5}, \ref{lm6} and \ref{lm7} .
\end{proof}
\begin{remark}
 Note that all the results are valid uniformly in $\kappa$ as $\kappa\to 0$. Therefore, present analysis provides also convergence of unsteady Navier-Stokes equation to its steady state system.
\end{remark}
\section{Conclusion}
In this article, the convergence of the Kelvin-Voigt viscoelastic fluid flow model to a steady state is established. Both exponential and power convergence results are proved under different prescribed conditions on the forcing function. Moreover, all the results are valid uniformly in the time relaxation or regularizing parameter $\kappa$ as $\kappa\to 0$, establishing validity even for the Navier-Stokes system.

\noindent
{\bf{Acknowledgements.}}
The second author acknowledges the support provided by the National Programme on Differential Equations:
Theory, Computation and Applications (NPDE-TCA) vide the DST project No.SR/S4/MS:639/90. The first  
author acknowledges the financial  support from UGC, Govt. India.

\bibliographystyle{amsplain}

\end{document}